\documentclass[10pt]{amsart}

\newcounter{mnote}
\let\oldmarginpar\marginpar
\renewcommand\marginpar[1]{\-\oldmarginpar[\raggedleft\footnotesize #1]%
{\raggedright\footnotesize #1}}

\usepackage[english]{babel}
\usepackage{verbatim}

\usepackage{amssymb}
\usepackage{amsthm}
\usepackage{amsmath,amscd}
\usepackage{mathrsfs}
\usepackage{stmaryrd}
\usepackage{chemarrow}
\usepackage[norelsize]{algorithm2e}
\usepackage{enumerate}
\usepackage{graphicx}
\usepackage[all]{xy}
\usepackage{tikz}
\usetikzlibrary{arrows}
\usepackage{multirow}
\usepackage[colorinlistoftodos]{todonotes}
\usepackage[colorlinks=true, allcolors=blue]{hyperref}
\usepackage[hyperpageref]{backref}

\newtheorem{theorem}{Theorem}[section]
\newtheorem{lemma}[theorem]{Lemma}
\newtheorem{corollary}[theorem]{Corollary}

\newtheorem{remark}[theorem]{Remark}

\newcommand{\curl}{\operatorname{curl}}
\renewcommand{\div}{\operatorname{div}}
\newcommand{\grad}{\operatorname{grad}}

\numberwithin{equation}{section}

\begin{document}
\title[Robust Mixed Methods for Quad-Curl Perturbation Problem]{Robust Mixed Finite Element Methods for a quad-curl singular perturbation problem}
\author{Xuehai Huang}%
\address{School of Mathematics, Shanghai University of Finance and Economics, Shanghai 200433, China}%
\email{huang.xuehai@sufe.edu.cn}%
\author{Chao Zhang}%
\address{School of Mathematics, Shanghai University of Finance and Economics, Shanghai 200433, China}%
\email{zcmath@163.sufe.edu.cn}%

\thanks{The first author is the corresponding author. The first author was supported by the National Natural Science Foundation of China Project 12171300, and the Natural Science Foundation of Shanghai 21ZR1480500.}

\makeatletter
\@namedef{subjclassname@2020}{\textup{2020} Mathematics Subject Classification}
\makeatother

\begin{abstract}
Robust mixed finite element methods are developed for a quad-curl singular perturbation problem.
Lower order $H(\grad \curl)$-nonconforming but $H(\curl)$-conforming finite elements are constructed, which are extended to nonconforming finite element Stokes complexes and the associated commutative diagrams.
Then $H(\grad \curl)$-nonconforming finite elements are employed to discretize the quad-curl singular perturbation problem, which possess the sharp and uniform error estimates with respect to the perturbation parameter.
The Nitsche’s technique is exploited to achieve the optimal convergence rate in the case of the boundary layers. Numerical results are provided to verify the theoretical convergence rates.
In addition, the regularity of the quad-curl singular perturbation problem is established.
\end{abstract}

\maketitle


\section{Introduction}

Let $\Omega\subset\mathbb{R}^{3}$ be a bounded polyhedral domain. In this work, we shall develop robust finite element methods for the quad-curl singular perturbation problem
\begin{equation}\label{quadcurl}
\begin{aligned}
&\left\{\begin{aligned}
\varepsilon^{2}\curl^{4} \boldsymbol{u}+ \curl^{2} \boldsymbol{u}&=\boldsymbol{f} & & \text { in } \Omega, \\
\div\boldsymbol{u} &=0 & & \text { in } \Omega, \\
\boldsymbol{u} \times \boldsymbol{n}=\curl \boldsymbol{u} &=\mathbf{0} & & \text { on } \partial \Omega,
\end{aligned}\right.
\end{aligned}
\end{equation}
where $\boldsymbol{f} \in \boldsymbol{H}(\div, \Omega)$ with $\div \boldsymbol{f} = 0$ and positive perturbation parameter $\varepsilon$ has an upper bound. In particular, we are interested in the case $\varepsilon$ approaching zero, for which the boundary layer phenomenon appears.
The wellposedness of problem~\eqref{quadcurl} is closely related to the following Stokes complex in three dimensions
\begin{equation}\label{eq:StokesComplex3d}
0 \xrightarrow{\subset} H_0^{1}(\Omega) \xrightarrow{\nabla} \boldsymbol{H}_0(\grad\curl, \Omega) \xrightarrow{\curl} \boldsymbol{H}_0^{1}(\Omega; \mathbb{R}^{3}) \xrightarrow{\div} L_0^{2}(\Omega) \rightarrow 0,
\end{equation}
where $\boldsymbol{H}_0(\operatorname{grad} \curl, \Omega):=\left\{\boldsymbol{v} \in \boldsymbol{H}_0(\curl, \Omega): \curl \boldsymbol{v} \in \boldsymbol{H}_0^{1}\left( \Omega; \mathbb{R}^{3}\right)\right\}$.

To discretize the qual-curl problem, $H(\grad\curl)$-conforming finite elements are constructed in \cite{ZhangZhang2020,HuZhangZhang2020,ZhangWangZhang2019,ChenHuang2022}, while the conforming finite elements in three dimensions suffer from the large number of degrees of freedom (DoFs).
To lower the number of DoFs, some nonstandard discretizations of Stokes complex \eqref{eq:StokesComplex3d} are advanced, for example, conforming macro-element Stokes complex on Alfeld split meshes in~\cite{HuZhangZhang2022} and conforming virtual element Stokes complex in \cite{BeiraodaVeigaDassiVacca2020}. 
In addition, the nonconforming finite element Stokes complexes in \cite{Huang2020,ZhengHuXu2011} have rather lower dimenions, for example
the $H(\grad\curl)$-nonconforming finite element in \cite{Huang2020} has only $14$ DoFs.
We refer to~\cite{CaoChenHuang2022} for a decoupled finite element method for the quad-curl problem.

Although the $H(\grad \curl)$-nonconforming elements in \cite{Huang2020,ZhengHuXu2011} solve the qual-curl problem correctly and have optimal convergence rates, the convergence rates will deteriorate when they were applied to discretize the double curl problem, i.e. problem \eqref{quadcurl} with $\varepsilon=0$, since these nonconforming elements are also $H(\curl)$-nonconforming; see Section 7.9 in \cite{BoffiBrezziFortin2013}.
We numerically test the Huang element in~\cite{Huang2020} for the double curl problem in Section~\ref{sec:numericalresults}, and do not observe any convergence, similarly as the Morley element method for Poisson equation \cite[Section 3]{NilssenTaiWinther2001}.
Hence the convergence would also deteriorate if using these nonconforming finite elements to discretize the quad-curl singular perturbation problem \eqref{quadcurl} with very small $\varepsilon$.

To this end, we will construct $H(\grad\curl)$-nonconforming but $H(\curl)$-conforming finite elements for problem \eqref{quadcurl} in this paper. For tetrahedron $K$, the space of shape functions $
\boldsymbol{W}_{k}(K):=\nabla \mathbb{P}_{k+1}(K) \oplus \boldsymbol{x} \times \mathbb{P}_{1}\left(K; \mathbb{R}^{3}\right)\oplus b_{K}\mathbb{P}_{1}(K; \mathbb{R}^{3})$ for $k=1,2$ is an enirchment of  N\'ed\'elec finite element space $\boldsymbol{V}_{k}^{ND}(K)=\nabla \mathbb{P}_{k+1}(K) \oplus \boldsymbol{x} \times \mathbb{P}_{1}\left(K; \mathbb{R}^{3}\right)$ \cite{Nedelec1980,Nedelec1986} with bubble functions. The dimension of $\boldsymbol{W}_{1}(K)$ is 32.
The DoFs for space $\boldsymbol{W}_{k}(K)$ are inherited from those of N\'ed\'elec elements
\begin{align*}
\int_{e} \boldsymbol{v} \cdot \boldsymbol{t}\, q\, \mathrm{d} s & \quad \forall~q \in \mathbb{P}_{k}(e), e \in \mathcal{E}(K),\\
\int_{F}\boldsymbol{v} \times \boldsymbol{n}\cdot \boldsymbol{q}\, \mathrm{d}S &\quad \forall~\boldsymbol{q}\in \boldsymbol{RM}_{k-2}(F), F \in \mathcal{F}(K),\\
\int_{F}(\curl \boldsymbol{v}) \times \boldsymbol{n}\cdot \boldsymbol{q}\, \mathrm{d}S &\quad \forall~\boldsymbol{q}\in \boldsymbol{RT}(F), F \in \mathcal{F}(K).
\end{align*}
After constructing the global $H(\grad \curl)$-nonconforming space $\boldsymbol{W}_{h0}$, we derive the nonconforming finite element discretization of Stokes complex \eqref{eq:StokesComplex3d}
\begin{equation*}
0 \xrightarrow{\subset} V_{h}^{g} \xrightarrow{\nabla} \boldsymbol{W}_{h0} \xrightarrow{\curl} \boldsymbol{V}_{h0}^{d} \xrightarrow{\div} \mathcal{Q}_{h} \rightarrow 0,
\end{equation*}
where $V_{h}^{g}$ is the $(k+1)$th order Lagrange element space, $\boldsymbol{V}_{h0}^d$ is the $H^1$-nonconforming finite element space in \cite{TaiWinther2006}, and $\mathcal{Q}_{h}$
is the piecewise constant space. The associated commutative diagrams are also established with lower smooth interpolations.

With the $H(\grad \curl)$-nonconforming space $\boldsymbol{W}_{h0}$ and the Lagrange element space $V_{h}^{g}$, we propose a robust mixed finite element method for the quad-curl singular perturbation problem~\eqref{quadcurl}.
The discrete Poincar\'e inequality and the discrete inf-sup condition are proved.
In addition, we create the regularity of the quad-curl singular perturbation problem~\eqref{quadcurl} with the help of the regularity of the Stokes equation under the assumption domain $\Omega$ is convex.
After deriving interpolation error estimates in detail,
we achieve the optimal convergence rate $O(h)$ of the energy error for any fixed $\varepsilon$, and the uniform convergence rate $O(h^{1/2})$ of the energy error with respect to $\varepsilon$ in consideration of the boundary layers.

The uniform convergence rate $O(h^{1/2})$ is sharp but apparently not optimal.
Then we modify the previous mixed finite element method by applying the Nitsche's technique \cite{Nitsche1971,Schieweck2008,GuzmanLeykekhmanNeilan2012}. The resulting discrete method has the uniform convergence rate $O(h^{2})$ of the energy error when $\varepsilon$ approaches zero, which is optimal.
Numerical experiments are provided to confirm these uniform convergence rates.

The rest of this paper is organized as follows. We construct a nonconforming finite element Stokes complex in Section~\ref{sec:femstokescomplex}. The regularity of the quad-curl singular perturbation problem is discussed in Section~\ref{sec:regularity}. In Section~\ref{sec:mfem} we propose and analyze a robust mixed finite element method.
A modified mixed finite element method based on Nitsche's technique is devised in Section~\ref{sec:modifiedmfem}. Some numerical results are shown in Section~\ref{sec:numericalresults}.


\section{Nonconforming Finite Element Stokes Complexes}\label{sec:femstokescomplex}

In this section we will construct two $H(\grad \curl)$-nonconforming but $H(\curl)$-conforming finite elements, and present the corresponding finite element Stokes complexes.

\subsection{Notation}

Given integer $m\geq0$ and a bounded domain $D\subset \mathbb{R}^{d}$ with integer $d\geq1$,  let $H^{m}(D)$ be the space of all square-integrable functions whose distributional derivatives up to order $m$ are also square-integrable. The norm and semi-norm are denoted by $\|\cdot \|_{m, D}$ and $|\cdot |_{m, D}$, respectively. 
Set $L^2(D)=H^{0}(D)$ with inner product $(\cdot, \cdot)_D$.
When $D$ is $\Omega$, $\|\cdot \|_{m, D}$, $|\cdot |_{m, D}$ and $(\cdot, \cdot)_D$ will be abbreviated as $\|\cdot \|_{m}$, $|\cdot |_{m}$ and $(\cdot, \cdot)$.
Denote by $H^{m}_{0}(D)$ the closure of $C^{\infty}_{0}(D)$ with respect to the norm $\|\cdot \|_{m,D}$.
Let $\mathbb{P}_{k}(D)$ stand for the set of all polynomials in $D$ with the total degree no more than $k$. Denote the vector version of $H^{m}(D)$, $H_0^{m}(D)$ and $\mathbb{P}_{k}(D)$ by $\boldsymbol{H}^{m}(D; \mathbb{R}^{d})$, $\boldsymbol{H}_0^{m}(D; \mathbb{R}^{d})$ and $\mathbb{P}_{k}(D; \mathbb{R}^{d})$, respectively. 
Let $Q_D^{k}$ be the $L^{2}$-orthogonal projector onto space $\mathbb{P}_{k}(D)$ or $\mathbb{P}_{k}(D;\mathbb R^d)$. 

The gradient operator, curl operator and divergence operator are denoted by $\nabla$, $\curl$ and $\div$, respectively. 
Let $\boldsymbol{H}(\curl, D)$ be the space of all functions $\boldsymbol{v}$ in $\boldsymbol{L}^2(D;\mathbb R^3)$ satisfying $\curl\boldsymbol{v}\in\boldsymbol{L}^2(D;\mathbb R^3)$, and $\boldsymbol{H}(\div, D)$ be the space of all functions $\boldsymbol{v}$ in $\boldsymbol{L}^2(D;\mathbb R^3)$ satisfying $\div\boldsymbol{v}\in L^2(D)$. The associated spaces with the vanishing trace are denoted by $\boldsymbol{H}_{0}(\curl, D)$ and $\boldsymbol{H}_{0}(\div, D)$.
Let 
\[
\boldsymbol{H}(\grad\curl,D):=\{\boldsymbol{v}\in\boldsymbol{H}(\curl, D):\curl\boldsymbol{v}\in\boldsymbol{H}^{1}(D;\mathbb R^3)\},
\]
\[
\boldsymbol{H}_0(\operatorname{grad} \curl, D):=\left\{\boldsymbol{v} \in \boldsymbol{H}_0(\curl, D): \curl \boldsymbol{v} \in \boldsymbol{H}_0^{1}\left(D; \mathbb{R}^{3}\right)\right\},
\]
\[
\boldsymbol{H}^{1}(\curl,D):=\{\boldsymbol{v}\in\boldsymbol{H}^{1}(D;\mathbb R^3):\curl\boldsymbol{v}\in\boldsymbol{H}^{1}(D;\mathbb R^3)\},
\]
\[
\boldsymbol{H}_0^{1}(\curl,D):=\{\boldsymbol{v}\in\boldsymbol{H}_0^{1}(D;\mathbb R^3):\curl\boldsymbol{v}\in\boldsymbol{H}_0^{1}(D;\mathbb R^3)\}.
\]
Denote by $L_{0}^{2}(D)$ the subspace of $L^{2}(D)$ having zero mean value.

Let $\{\mathcal{T}_{h}\}_{h>0}$ be a regular family of tetrahedral meshes of $\Omega$. Denote by $\mathcal{F}_{h}$, $\mathcal{F}_{h}^{i}$, $\mathcal{E}_{h}$, $\mathcal{E}_{h}^i$, $\mathcal{V}_{h}$ and $\mathcal{V}_{h}^i$ the set of all faces, interior faces, edges, interior edges, vertices and interior vertices of the partition $\mathcal{T}_{h}$, respectively. For each tetrahedron $K \in \mathcal{T}_{h}$, let $\boldsymbol{n}_{K}$ be the unit outward normal vector to $\partial K$, which will be abbreviated as $\boldsymbol{n}$ if not causing any confusion. We fix a unit normal vector $\boldsymbol{n}_{F}$ for each face $F \in \mathcal{F}_{h}$, and a unit tangential vector $\boldsymbol{t}_e$ for each edge $e \in \mathcal{E}_{h}$. We also abbreviate $\boldsymbol{n}_{F}$ and $\boldsymbol{t}_e$ as $\boldsymbol{n}$ and $\boldsymbol{t}$ accordingly when not causing any confusion.
For $K \in \mathcal{T}_{h}$, denote by $\mathcal{F}(K), \mathcal{E}(K)$ and $\mathcal{V}(K)$ the set of all faces, edges and vertices of $K$, respectively. 
For $F \in \mathcal{F}_{h}$, let $\mathcal{E}(F)$ be the set of all edges of $F$. And for $e \in \mathcal{E}(F)$, denote by $\boldsymbol{n}_{F,e}$  the unit vector being parallel to $F$ and outward normal to $\partial F$. Set $\boldsymbol t_{F,e}:=\boldsymbol{n}_{F}\times \boldsymbol{n}_{F,e}$, where $\times$ is the exterior product. For elementwise smooth function $\boldsymbol{v}$, define
\[
\|\boldsymbol{v}_{h}\|_{1,h}^{2}:= \sum_{K \in \mathcal{T}_{h}}\|\boldsymbol{v}_{h}\|_{1,K}^{2},\quad |\boldsymbol{v}_{h}|_{1,h}^{2}:= \sum_{K \in \mathcal{T}_{h}}|\boldsymbol{v}_{h}|_{1,K}^{2}.
\]
Let $\nabla_{h}$ be the elementwise counterpart of $\nabla$ with respect to $\mathcal{T}_{h}$.

\subsection{$H(\grad\curl)$-nonconforming finite elements}
We focus on constructing $H(\grad\curl)$-nonconforming but $H(\curl)$-conforming finite elements in this subsection. 

For tetrahedron $K$, let 
\[
\boldsymbol{V}_{k}^{ND}(K):=\nabla \mathbb{P}_{k+1}(K) \oplus \boldsymbol{x} \times \mathbb{P}_{1}\left(K ; \mathbb{R}^{3}\right) \quad \text { for } k=1,2.
\]
Then $\boldsymbol{V}_{1}^{ND}(K)$ is the first order N\'ed\'elec finite element space of the first kind~\cite{Nedelec1980}, and $\boldsymbol{V}_{2}^{ND}(K)$ is the second order N\'ed\'elec finite element space of the second kind~\cite{Nedelec1986}.
Now enirch $\boldsymbol{V}_{k}^{ND}(K)$ with bubble functions to define the space of shape functions for $H(\grad\curl)$-nonconforming finite elements as
\[
\boldsymbol{W}_{k}(K):=\boldsymbol{V}_{k}^{ND}(K)\oplus b_{K}\mathbb{P}_{1}(K;\mathbb{R}^{3}) \quad \text { for } k=1,2,
\]
where $b_K := \lambda_{1}\lambda_{2}\lambda_{3}\lambda_{4}$ is the bubble function of $K$, and $\lambda_{i}$ ($i=1,2,3,4$) are the barycentric coordinates corresponding to the vertices of $K$.
Clearly $\mathbb{P}_{k}(K;\mathbb R^3)\subseteq\boldsymbol{V}_{k}^{ND}(K)\subset\boldsymbol{W}_{k}(K)$.
We have
\begin{equation*}
\operatorname{dim} \boldsymbol{W}_{k}(K)= \begin{cases}32, & k=1, \\ 42, & k=2.\end{cases}
\end{equation*}
We propose the following degrees of freedom (DoFs) for space $\boldsymbol{W}_{k}(K)$
\begin{align}
\label{Wkdof1}
\int_{e} \boldsymbol{v} \cdot \boldsymbol{t}\, q\, \mathrm{d} s & \quad \forall~q \in \mathbb{P}_{k}(e), e \in \mathcal{E}(K),\\
\label{Wkdof2}
\int_{F}\boldsymbol{v} \times \boldsymbol{n}\cdot \boldsymbol{q}\, \mathrm{d}S &\quad \forall~\boldsymbol{q}\in \boldsymbol{RM}_{k-2}(F), F \in \mathcal{F}(K),\\
\label{Wkdof3}
\int_{F}(\curl \boldsymbol{v}) \times \boldsymbol{n}\cdot \boldsymbol{q}\, \mathrm{d}S &\quad \forall~\boldsymbol{q}\in \boldsymbol{RT}(F), F \in \mathcal{F}(K),
\end{align}
where $\boldsymbol{RM}_{k-2}(F):=\mathbb{P}_{0}(F;\mathbb{R}^{2})\oplus\boldsymbol{x}^{\perp}\mathbb{P}_{k-2}(F)$ with $\boldsymbol{x}^{\perp}:=(\boldsymbol{n}\times \boldsymbol{x})|_F$, and $\boldsymbol{RT}(F):=\mathbb{P}_{0}(F; \mathbb{R}^{2})\oplus\boldsymbol{x}\mathbb{P}_{0}(F)$.
DoFs \eqref{Wkdof1}-\eqref{Wkdof2} are inspired by the N\'ed\'elec elements. Indeed
DoFs \eqref{Wkdof1}-\eqref{Wkdof2} are uni-solvent for $\boldsymbol{V}_{k}^{ND}(K)$ \cite{Nedelec1980,Nedelec1986}. DoF \eqref{Wkdof3} is related to the bubble space $b_{K}\mathbb{P}_{1}(K;\mathbb{R}^{3})$.


To show the unisolvence of the $H(\grad\curl)$-nonconforming finite elements, we recall an $H^1$-nonconforming finite element for Stokes problem in \cite{TaiWinther2006}.
Let $\boldsymbol{V}^{d}(K)=\mathbb{P}_{1}\left(K ; \mathbb{R}^{3}\right)\oplus \curl (b_{K}\mathbb{P}_{1}(K; \mathbb{R}^{3}))$. The unisolvent DoFs for $\boldsymbol{V}^{d}(K)$ are (cf. \cite[Lemma 4]{TaiWinther2006})
\begin{align}
\label{Vskdof1}
\int_{F} (\boldsymbol{v} \cdot \boldsymbol{n})\, q\, \mathrm{d}S & \quad \forall~q \in \mathbb{P}_{1}(F), F \in \mathcal{F}(K),\\
\label{Vskdof2}
\int_F(\boldsymbol{v} \times \boldsymbol{n})\cdot \boldsymbol{q}\, \mathrm{d}S &\quad \forall~\boldsymbol{q}\in \boldsymbol{RT}(F), F \in \mathcal{F}(K).
\end{align}
DoF \eqref{Vskdof2} corresponds to DoF \eqref{Wkdof3}.

\begin{lemma}\label{lem:bkq0}
For $\boldsymbol{v}\in \boldsymbol{x} \times \mathbb{P}_{1}\left(K ; \mathbb{R}^{3}\right) + b_{K}\mathbb{P}_{1}(K;\mathbb{R}^{3})$ satisfying
$
\curl\boldsymbol{v}=\boldsymbol{0},
$
we have $\boldsymbol{v}=\boldsymbol{0}$.
\end{lemma}
\begin{proof}
Let $\boldsymbol{v}=\boldsymbol{\phi}+b_K\boldsymbol{\psi}$ with $\boldsymbol{\phi}\in\boldsymbol{x} \times \mathbb{P}_{1}(K ; \mathbb{R}^{3})$ and $\boldsymbol{\psi}\in\mathbb{P}_{1}(K;\mathbb{R}^{3})$. Then $\curl\boldsymbol{\phi}=-\curl(b_K\boldsymbol{\psi})\in\mathbb{P}_{1}(K;\mathbb{R}^{3})$, and
\[
(\curl\boldsymbol{\phi})\cdot\boldsymbol{n}=-\curl(b_K\boldsymbol{\psi})\cdot\boldsymbol{n}=-(\boldsymbol{n}\times\nabla)\cdot(b_K\boldsymbol{\psi})=0  \textrm{ on } \partial K.
\]
By the unisolvence of the lowest order Brezzi-Douglas-Marini (BDM) element \cite{BrezziDouglasDuranFortin1987}, we get $\curl\boldsymbol{\phi}=\boldsymbol{0}$ and $\curl(b_K\boldsymbol{\psi})=\boldsymbol{0}$. Noting that $\curl: \boldsymbol{x} \times \mathbb{P}_{1}(K ; \mathbb{R}^{3})\to\mathbb{P}_{1}(K ; \mathbb{R}^{3})$ is injective, it follows $\boldsymbol{\phi}=\boldsymbol{0}$.

Since $\curl(b_{K}\boldsymbol{\psi})=\boldsymbol{0}$ and $b_{K}\boldsymbol{\psi}\in\boldsymbol{H}_0^1(K;\mathbb R^3)$, by the de Rham complex \cite{CostabelMcIntosh2010}, there exists $w\in H_0^2(K)$ satisfying $b_{K}\boldsymbol{\psi}=\nabla w$. It is easy to see $w\in H_0^2(K)\cap\mathbb P_6(K)$, which means $w$ has a factor $b_K^2\in\mathbb P_8(K)$. Thus $w=0$ and $\boldsymbol{\psi}=\boldsymbol{0}$.
\end{proof}

As an immediate result of Lemma~\ref{lem:bkq0},
\begin{equation}\label{eq:WKcurl0}
\boldsymbol{W}_{k}(K)\cap\ker(\curl)=\boldsymbol{V}_{k}^{ND}(K)\cap\ker(\curl)=\nabla \mathbb{P}_{k+1}(K),
\end{equation}
where $\ker(\curl)$ means the kernel space of operator $\curl$.

\begin{lemma}
The degrees of freedom \eqref{Wkdof1}-\eqref{Wkdof3} are unisolvent for space $\boldsymbol{W}_{k}(K)$.
\end{lemma}
\begin{proof}
The number of degrees of freedom \eqref{Wkdof1}-\eqref{Wkdof3} is 
\[
6(k+1)+20+4(k-1)=10k+22,
\]
which equals to the dimension of $\boldsymbol{W}_{k}(K)$. 

Assume $\boldsymbol{v} \in \boldsymbol{W}_{k}(K)$ and all the DoFs \eqref{Wkdof1}-\eqref{Wkdof3} vanish.
Since $\boldsymbol{v}|_{\partial K}\in \boldsymbol{V}_{k}^{ND}(K)|_{\partial K}$ and DoFs \eqref{Wkdof1}-\eqref{Wkdof2} are unisolvent for $\boldsymbol{V}_{k}^{ND}(K)$, we obtain $\boldsymbol{v}\times \boldsymbol{n}=\boldsymbol{0}$ on $\partial K$. As a result, 
\[
(\curl\boldsymbol{v}) \cdot \boldsymbol{n}=(\boldsymbol{n}\times\nabla)\cdot \boldsymbol{v}=0 \quad \textrm{ on } \partial K.
\]
Noting that $\curl\boldsymbol{v}\in\boldsymbol{V}^{d}(K)$ and DoFs \eqref{Vskdof1}-\eqref{Vskdof2} are unisolvent for $\boldsymbol{V}^{d}(K)$, we get from the last identity and the vanishing DoF \eqref{Wkdof3} that $\curl\boldsymbol{v}=\boldsymbol{0}$. Due to~\eqref{eq:WKcurl0}, $\boldsymbol{v}\in\nabla \mathbb{P}_{k+1}(K)\subset \boldsymbol{V}_{k}^{ND}(K)$. Therefore $\boldsymbol{v}=\boldsymbol{0}$ follows from the unisolvence of N\'ed\'elec elements.
\end{proof}

Define the global $H(\grad \curl)$-nonconforming
element space
\[
\boldsymbol{W}_{h0}:=\{\boldsymbol{v}_{h} \in \boldsymbol{W}_{h}:\text{ DoF \eqref{Wkdof3} on } \partial \Omega \text{ vanishes}\},
\]
where
\begin{align*}
\boldsymbol{W}_{h}:=\{\boldsymbol{v}_{h} \in \boldsymbol{L}^{2}(\Omega ; \mathbb{R}^{3}):&\,\boldsymbol{v}_{h}|_{K} \in \boldsymbol{W}_{k}(K) \text { for each } K \in \mathcal{T}_{h}, 
\\& \text { all the DoFs \eqref{Wkdof1}-\eqref{Wkdof3} are single-valued}, \\
&\text{ and all the DoFs \eqref{Wkdof1}-\eqref{Wkdof2} on } \partial \Omega \text{ vanish}\}.
\end{align*}
By definition, all the DoFs \eqref{Wkdof1}-\eqref{Wkdof3} for functions in $\boldsymbol{W}_{h0}$ vanish on $\partial \Omega$.
Noting that \eqref{Wkdof1}-\eqref{Wkdof2} are exactly the unisolvent DoFs for the N\'ed\'elec element space $\boldsymbol{V}_{k}^{ND}(K)$, by $\boldsymbol{W}_{k}(K)|_{\partial K}=\boldsymbol{V}_{k}^{ND}(K)|_{\partial K}$, we have $\boldsymbol{W}_{h}\subset \boldsymbol{H}_{0}(\curl, \Omega)$. 




\subsection{Finite element Stokes complexes}
We will consider the nonconforming finite element discretization of the Stokes complex \eqref{eq:StokesComplex3d} in this section. 
Recall the Lagrange element space
\[
V_{h}^{g}:=\left\{v_{h} \in H_0^{1}(\Omega):\left.v_{h}\right|_{K} \in \mathbb{P}_{k+1}(K) \text{ for each } K \in \mathcal{T}_{h}\right\},
\]
the piecewise constant space
\[
\mathcal{Q}_{h}:=\{q_{h} \in L_0^{2}(\Omega):q_{h}|_{K} \in \mathbb{P}_{0}(K) \text { for each } K \in \mathcal{T}_{h}\},
\]
and the $H^1$-nonconforming finite element space in \cite{TaiWinther2006}
\[
\boldsymbol{V}_{h0}^d:=\{\boldsymbol{v}_{h} \in \boldsymbol{V}_{h}^d: \text{ all the DoFs \eqref{Vskdof1}-\eqref{Vskdof2} on } \partial \Omega \text{ vanish}\},
\]
where
\begin{align*}
\boldsymbol{V}_{h}^d:=\{\boldsymbol{v}_{h} \in \boldsymbol{H}_0(\div,\Omega):&\,\boldsymbol{v}_{h}|_{K} \in \boldsymbol{V}^d(K) \text { for each } K \in \mathcal{T}_{h}, 
\\& \text { all the DoFs \eqref{Vskdof1}-\eqref{Vskdof2} are single-valued} \}.
\end{align*}
Thanks to the weak continuity,
the space $\boldsymbol{V}^{d}_{h0}$ possesses the norm equivalence \cite{Brenner2003}
\begin{equation*}
\|\boldsymbol{v}_{h}\|_{1, h}\simeq |\boldsymbol{v}_{h}|_{1, h} \quad \forall~\boldsymbol{v}_{h} \in \boldsymbol{V}_{h0}^{d}.
\end{equation*}

\begin{lemma} 
The nonconforming finite element Stokes complexes
\begin{equation}\label{discretecomplex}
0 \xrightarrow{\subset} V_{h}^{g} \xrightarrow{\nabla} \boldsymbol{W}_{h0} \xrightarrow{\curl} \boldsymbol{V}_{h0}^{d} \xrightarrow{\div} \mathcal{Q}_{h} \rightarrow 0,
\end{equation}
\begin{equation}\label{discretecomplex1}
0 \xrightarrow{\subset} V_{h}^{g} \xrightarrow{\nabla} \boldsymbol{W}_{h} \xrightarrow{\curl} \boldsymbol{V}_{h}^{d} \xrightarrow{\div} \mathcal{Q}_{h} \rightarrow 0
\end{equation}
are exact.
\end{lemma}
\begin{proof}
For the proof of \eqref{discretecomplex1} is similar, we only prove \eqref{discretecomplex} is an exact sequence.
It is straightforward to check that each space in sequence \eqref{discretecomplex}  is mapped into the succeeding space by the given differential operator, hence sequence \eqref{discretecomplex} is a complex. We refer to \cite{TaiWinther2006} for $\div\boldsymbol{V}_{h0}^{d}= \mathcal{Q}_{h}$.

It follows from Stokes complex \eqref{eq:StokesComplex3d} and \eqref{eq:WKcurl0} that
$\boldsymbol{W}_{h0} \cap \ker(\curl)=\nabla V_{h}^{g}$. Then
\begin{align*}
\dim\curl\boldsymbol{W}_{h0} &=\dim\boldsymbol{W}_{h0}-\dim V_{h}^{g} \\
&=(k+1)\#\mathcal{E}_{h}^i +(k+4)\# \mathcal{F}_{h}^i-\# \mathcal{V}_{h}^i-k\# \mathcal{E}_{h}^i-(k-1)\# \mathcal{F}_{h}^i \\
&=\# \mathcal{E}_{h}^i+5 \# \mathcal{F}_{h}^i-\# \mathcal{V}_{h}^i,
\end{align*}
By $\div\boldsymbol{V}_{h0}^{d}= \mathcal{Q}_{h}$,
\[
\dim\boldsymbol{V}_{h0}^{d}\cap\ker(\div)=\dim\boldsymbol{V}_{h0}^{d}-\dim\mathcal{Q}_{h}=6\# \mathcal{F}_{h}^i-\# \mathcal{T}_{h}+1.
\]
Employing the Euler's formula, we get
\[
\dim\boldsymbol{V}_{h0}^{d}\cap\ker(\div)-\dim\curl\boldsymbol{W}_{h0}=\#\mathcal{V}_{h}^i-\# \mathcal{E}_{h}^i+\# \mathcal{F}_{h}^i-\# \mathcal{T}_{h}+1=0,
\]
which indicates $\boldsymbol{V}_{h0}^{d}\cap\ker(\div)=\curl\boldsymbol{W}_{h0}$.
\end{proof}

\subsection{Interpolations and commutative diagram}
To introduce interpolations for complexes \eqref{discretecomplex}-\eqref{discretecomplex1}, we recall the finite element de Rham complex \cite{Arnold2018}
\[
0 \xrightarrow{\subset} V_{h}^{g} \xrightarrow{\nabla} \boldsymbol{V}_{h}^{ND} \xrightarrow{\curl} \boldsymbol{V}_{h}^{BDM} \xrightarrow{\div} \mathcal{Q}_{h} \rightarrow 0,
\]
where
\[
\boldsymbol{V}_{h}^{ND}:=\{\boldsymbol{v}_{h} \in \boldsymbol{H}_0(\curl,\Omega): \boldsymbol{v}_{h}|_{K} \in\boldsymbol{V}_{k}^{ND}(K) \text{ for each } K \in \mathcal{T}_{h}\},
\]
\[
\boldsymbol{V}_{h}^{BDM}:=\{\boldsymbol{v}_{h} \in \boldsymbol{H}_0(\div,\Omega): \boldsymbol{v}_{h}|_{K} \in\mathbb{P}_{1}\left(K ; \mathbb{R}^{3}\right) \text{ for each } K \in \mathcal{T}_{h}\}.
\]
Let $I_h^g: L^2(\Omega)\to V_{h}^{g}$, $I_h^{ND}: \boldsymbol L^2(\Omega;\mathbb R^3)\to \boldsymbol{V}_{h}^{ND}$, $I_h^{BDM}: \boldsymbol L^2(\Omega;\mathbb R^3)\to \boldsymbol{V}_{h}^{BDM}$ and $Q_h: L^2(\Omega)\to \mathcal{Q}_{h}$ be the local bounded projections devised in \cite{FalkWinther2014}, then it holds the commutative diagram
\begin{equation}\label{eq:feec}
\resizebox{0.9\hsize}{!}{$
\begin{array}{c}
\xymatrix{
 0 \ar[r]^-{\subset} & H_0^1(\Omega)\ar[d]^{I_h^g} \ar[r]^-{\nabla} & \boldsymbol{H}_0(\curl, \Omega) \ar[d]^{I_h^{ND}} \ar[r]^-{\curl}
                & \boldsymbol{H}_0(\div, \Omega) \ar[d]^{I_h^{BDM}}   \ar[r]^-{\div} & \ar[d]^{Q_{h}}L_0^2(\Omega) \ar[r]^{} & 0 \\
 0 \ar[r]^-{\subset} & V_{h}^g\ar[r]^-{\nabla} & \boldsymbol V_{h}^{ND} \ar[r]^{\curl}
                &  \boldsymbol V_{h}^{BDM}   \ar[r]^{\div} &  \mathcal Q_{h} \ar[r]^{}& 0  }
\end{array}.
$}
\end{equation}
We also have
\begin{equation}\label{eq:IhNDestimate}
\|\boldsymbol{v}-I_{h}^{ND}\boldsymbol{v}\|_{0}\lesssim h^{j}|\boldsymbol{v}|_j  \quad\forall~\boldsymbol{v}\in\boldsymbol{H}_0(\curl,\Omega)\cap\boldsymbol{H}^j(\Omega,\mathbb R^3), j=1,2,k+1,
\end{equation}
\begin{equation}\label{eq:IhBDMbound}
\|I_{h}^{BDM}\boldsymbol{v}\|_{0}\lesssim \|\boldsymbol{v}\|_{0}+h\|\div\boldsymbol{v}\|_{0}\quad\forall~\boldsymbol{v}\in\boldsymbol{H}_0(\div,\Omega),
\end{equation}
\begin{equation}\label{eq:IhBDMestimate}
\sum_{K\in\mathcal T_h}h_K^{2(i-j)}|\boldsymbol{v}-I_{h}^{BDM}\boldsymbol{v}|_{i,K}^2\lesssim |\boldsymbol{v}|_j^2  \quad\forall~\boldsymbol{v}\in\boldsymbol{H}_0(\div,\Omega)\cap\boldsymbol{H}^j(\Omega,\mathbb R^3), 
\end{equation}
for $i=0,1$ and $j=1,2$.

Based on DoFs \eqref{Wkdof1}-\eqref{Wkdof3},
define $I_{h}^{gc}: \boldsymbol L^2(\Omega;\mathbb R^3)\to \boldsymbol{W}_{h}$ as follows:
\begin{align*}
\big((I_{h}^{gc}\boldsymbol{v}) \cdot \boldsymbol{t}, q\big)_e=\big((I_{h}^{ND}\boldsymbol{v}) \cdot \boldsymbol{t}, q\big)_e & \quad \forall~q \in \mathbb{P}_{k}(e), e \in \mathcal{E}_h, \\
\big((I_{h}^{gc}\boldsymbol{v}) \times \boldsymbol{n}, \boldsymbol{q}\big)_F=\big((I_{h}^{ND}\boldsymbol{v}) \times \boldsymbol{n}, \boldsymbol{q}\big)_F &\quad \forall~\boldsymbol{q}\in \boldsymbol{RM}_{k-2}(F), F\in \mathcal{F}_h, \\
\big(\curl(I_{h}^{gc}\boldsymbol{v}) \times \boldsymbol{n}, \boldsymbol{q}\big)_F=\big(\{\curl(I_{h}^{ND}\boldsymbol{v}) \times \boldsymbol{n}\}, \boldsymbol{q}\big)_F &\quad \forall~\boldsymbol{q}\in \boldsymbol{RT}(F), F \in \mathcal{F}_h,
\end{align*}
where $\{\cdot\}$ is the average operator across $F$ for interior face $F\in\mathcal F_h^i$, and the identity operator for $F\in\mathcal F_h^{\partial}:=\mathcal F_h\backslash\mathcal F_h^i$. Similarly define $I_{h0}^{gc}: \boldsymbol L^2(\Omega;\mathbb R^3)\to \boldsymbol{W}_{h0}$ as follows:
\begin{align*}
\big((I_{h0}^{gc}\boldsymbol{v}) \cdot \boldsymbol{t}, q\big)_e=\big((I_{h}^{ND}\boldsymbol{v}) \cdot \boldsymbol{t}, q\big)_e & \quad \forall~q \in \mathbb{P}_{k}(e), e \in \mathcal{E}_h, \\
\big((I_{h0}^{gc}\boldsymbol{v}) \times \boldsymbol{n}, \boldsymbol{q}\big)_F=\big((I_{h}^{ND}\boldsymbol{v}) \times \boldsymbol{n}, \boldsymbol{q}\big)_F &\quad \forall~\boldsymbol{q}\in \boldsymbol{RM}_{k-2}(F), F\in \mathcal{F}_h, \\
\big(\curl(I_{h0}^{gc}\boldsymbol{v}) \times \boldsymbol{n}, \boldsymbol{q}\big)_F=\big(\{\curl(I_{h}^{ND}\boldsymbol{v}) \times \boldsymbol{n}\}, \boldsymbol{q}\big)_F &\quad \forall~\boldsymbol{q}\in \boldsymbol{RT}(F), F \in \mathcal{F}_h^i.
\end{align*}
Based on DoFs \eqref{Vskdof1}-\eqref{Vskdof2},
define $I_{h}^{d}: \boldsymbol L^2(\Omega;\mathbb R^3)\to \boldsymbol{V}_{h}^{d}$ as follows:
\begin{align*}
\big((I_{h}^{d}\boldsymbol{v}) \cdot \boldsymbol{n}, q\big)_F=\big((I_{h}^{BDM}\boldsymbol{v}) \cdot \boldsymbol{n}, q\big)_F &\quad \forall~q\in \mathbb P_{1}(F), F\in \mathcal{F}_h, \\
\big((I_{h}^{d}\boldsymbol{v}) \times \boldsymbol{n}, \boldsymbol{q}\big)_F=\big(\{(I_{h}^{BDM}\boldsymbol{v}) \times \boldsymbol{n}\}, \boldsymbol{q}\big)_F &\quad \forall~\boldsymbol{q}\in \boldsymbol{RT}(F), F \in \mathcal{F}_h.
\end{align*}
And define $I_{h0}^{d}: \boldsymbol L^2(\Omega;\mathbb R^3)\to \boldsymbol{V}_{h0}^{d}$ as
\begin{align*}
\big((I_{h0}^{d}\boldsymbol{v}) \cdot \boldsymbol{n}, q\big)_F=\big((I_{h}^{BDM}\boldsymbol{v}) \cdot \boldsymbol{n}, q\big)_F &\quad \forall~q\in \mathbb P_{1}(F), F\in \mathcal{F}_h, \\
\big((I_{h0}^{d}\boldsymbol{v}) \times \boldsymbol{n}, \boldsymbol{q}\big)_F=\big(\{(I_{h}^{BDM}\boldsymbol{v}) \times \boldsymbol{n}\}, \boldsymbol{q}\big)_F &\quad \forall~\boldsymbol{q}\in \boldsymbol{RT}(F), F \in \mathcal{F}_h^i.
\end{align*}

\begin{lemma}
We have the commutative diagrams
\begin{equation}\label{eq:stokescdncfem0}
\resizebox{0.9\hsize}{!}{$
\begin{array}{c}
\xymatrix{
 0 \ar[r]^-{\subset} & H_0^1(\Omega)\ar[d]^{I_h^{g}} \ar[r]^-{\nabla} & \boldsymbol{H}_0(\grad\curl, \Omega) \ar[d]^{I_{h0}^{gc}} \ar[r]^-{\curl}
                & \boldsymbol{H}_0^1(\Omega; \mathbb{R}^3) \ar[d]^{I_{h0}^d}   \ar[r]^-{\div} & \ar[d]^{Q_{h}}L_0^2(\Omega) \ar[r]^{} & 0 \\
 0 \ar[r]^-{\subset} & V_{h}^g\ar[r]^-{\nabla} & \boldsymbol W_{h0} \ar[r]^{\curl}
                &  \boldsymbol V_{h0}^d   \ar[r]^{\div} &  \mathcal Q_{h} \ar[r]^{}& 0  }
\end{array},
$}
\end{equation}
\begin{equation}\label{eq:stokescdncfem}
\resizebox{0.9\hsize}{!}{$
\begin{array}{c}
\xymatrix{
 0 \ar[r]^-{\subset} & H_0^1(\Omega)\ar[d]^{I_h^{g}} \ar[r]^-{\nabla} & \mathring{\boldsymbol{H}}(\grad\curl, \Omega) \ar[d]^{I_h^{gc}} \ar[r]^-{\curl}
                & \mathring{\boldsymbol{H}}^1(\Omega; \mathbb{R}^3) \ar[d]^{I_h^d}   \ar[r]^-{\div} & \ar[d]^{Q_{h}}L_0^2(\Omega) \ar[r]^{} & 0 \\
 0 \ar[r]^-{\subset} & V_{h}^g\ar[r]^-{\nabla} & \boldsymbol W_{h} \ar[r]^{\curl}
                &  \boldsymbol V_{h}^d   \ar[r]^{\div} &  \mathcal Q_{h} \ar[r]^{}& 0  }
\end{array},
$}
\end{equation}
where $\mathring{\boldsymbol{H}}(\grad\curl, \Omega):=\boldsymbol{H}(\grad\curl, \Omega)\cap\boldsymbol{H}_0(\curl, \Omega)$ and $\mathring{\boldsymbol{H}}^1(\Omega; \mathbb{R}^3):=\boldsymbol{H}^1(\Omega; \mathbb{R}^3)\cap\boldsymbol{H}_0(\div, \Omega)$.
\end{lemma}
\begin{proof}
We only prove \eqref{eq:stokescdncfem0}, since \eqref{eq:stokescdncfem} can be proved similarly.
For $v\in H_0^1(\Omega)$, by the commutative diagram \eqref{eq:feec}, it follows
\[
I_{h0}^{gc}(\nabla v)-\nabla(I_h^gv)=I_{h0}^{gc}(\nabla v)-I_h^{ND}(\nabla v)\in\boldsymbol W_{h0}.
\]
Noting that $\curl(I_{h}^{ND}(\nabla v))=I_{h}^{BDM}(\curl(\nabla v))=\boldsymbol{0}$, by the definition of $I_{h0}^{gc}$, all the DoFs \eqref{Wkdof1}-\eqref{Wkdof3} of $I_{h0}^{gc}(\nabla v)-I_h^{ND}(\nabla v)$ vanish, thus $I_{h0}^{gc}(\nabla v)-I_h^{ND}(\nabla v)=\boldsymbol{0}$, i.e.,
\begin{equation}\label{eq:stokescdncfemprop1}
I_{h0}^{gc}(\nabla v)=\nabla(I_h^gv)\quad\forall~v\in H_0^1(\Omega).
\end{equation}

For $\boldsymbol{v}\in \boldsymbol{H}_0(\grad\curl, \Omega)$, by complex \eqref{discretecomplex}, $I_{h0}^{d}(\curl \boldsymbol{v})-\curl(I_{h0}^{gc}\boldsymbol{v})\in\boldsymbol V_{h0}^d$.
We get from $I_{h}^{BDM}(\curl \boldsymbol{v})=\curl(I_{h}^{ND}\boldsymbol{v})$ and the definitions of $I_{h0}^{d}$ and $I_{h0}^{gc}$ that
\begin{align*}
\big((I_{h0}^{d}(\curl \boldsymbol{v})-\curl(I_{h0}^{gc}\boldsymbol{v})) \cdot \boldsymbol{n}, q\big)_F&=\big((I_{h}^{BDM}(\curl \boldsymbol{v})-\curl(I_{h0}^{gc}\boldsymbol{v})) \cdot \boldsymbol{n}, q\big)_F\\
&=\big((\curl(I_{h}^{ND}\boldsymbol{v}- I_{h0}^{gc}\boldsymbol{v})) \cdot \boldsymbol{n}, q\big)_F=0
\end{align*}
for $q\in \mathbb P_{1}(F)$ and $F\in \mathcal{F}_h$, and
\begin{align*}
&\quad\;\big((I_{h0}^{d}(\curl \boldsymbol{v})-\curl(I_{h0}^{gc}\boldsymbol{v})) \times \boldsymbol{n}, \boldsymbol{q}\big)_F\\
&=\big(\{(I_{h}^{BDM}(\curl \boldsymbol{v})) \times \boldsymbol{n}\}-\{\curl(I_{h}^{ND}\boldsymbol{v}) \times \boldsymbol{n}\}, \boldsymbol{q}\big)_F =0
\end{align*}
for all $\boldsymbol{q}\in \boldsymbol{RT}(F)$ and $F \in \mathcal{F}_h^i$.
Hence
\begin{equation}\label{eq:stokescdncfemprop2}
I_{h0}^{d}(\curl \boldsymbol{v})=\curl(I_{h0}^{gc}\boldsymbol{v}) \quad\forall~\boldsymbol{v}\in \boldsymbol{H}_0(\grad\curl, \Omega).
\end{equation}

For $\boldsymbol{v}\in \boldsymbol{H}_0^1(\Omega;\mathbb R^3)$, by complex \eqref{discretecomplex}, $Q_{h}(\div\boldsymbol{v})-\div(I_{h0}^{d}\boldsymbol{v})\in\mathcal Q_{h}$.
It follows from the definition of $I_{h0}^{d}$ and the integration by parts that
\[
(\div(I_{h0}^{d}\boldsymbol{v})-\div(I_{h}^{BDM}\boldsymbol{v}), q)_K=((I_{h0}^{d}\boldsymbol{v}-I_{h}^{BDM}\boldsymbol{v})\cdot\boldsymbol{n}, q)_{\partial K}=0
\]
for $q\in\mathbb P_0(K)$ and $K\in\mathcal T_h$. Thanks to the commutative diagram \eqref{eq:feec},
\begin{equation}\label{eq:stokescdncfemprop3}
\div(I_{h0}^{d}\boldsymbol{v})=\div(I_{h}^{BDM}\boldsymbol{v})=Q_{h}(\div\boldsymbol{v}).
\end{equation}
Finally combining \eqref{eq:stokescdncfemprop1}-\eqref{eq:stokescdncfemprop3} gives the commutative diagram \eqref{eq:stokescdncfem0}.
\end{proof}

\section{Regularity of the Quad-Curl Singular Perturbation Problem}\label{sec:regularity}

We will analyze the regularity of the quad-curl singular perturbation problem~\eqref{quadcurl} in this section. Hereafter we always assume domain $\Omega$ is convex.

\subsection{Mixed formulation}
Due to the identity $\curl^{2}\boldsymbol{v} = -\Delta\boldsymbol{v}+\nabla(\div\boldsymbol{v})$ with $\boldsymbol{v}=\curl\boldsymbol{u}$, the first equation in (\ref{quadcurl}) can be rewritten as
\begin{equation}\label{eq:curl2laplaceperturb}
-\varepsilon^{2} \curl \Delta(\curl \boldsymbol{u})+\curl^{2} \boldsymbol{u}=\boldsymbol{f}.
\end{equation}
Then a mixed formulation of problem \eqref{quadcurl} is to find $\boldsymbol{u} \in \boldsymbol{H}_{0}(\grad \curl, \Omega)$ and $\lambda \in H_{0}^{1}(\Omega)$ such that
\begin{align}
\label{mixfor1}
\varepsilon^{2}a(\boldsymbol{u}, \boldsymbol{v})+b(\boldsymbol{u}, \boldsymbol{v})+ c(\boldsymbol{v}, \lambda) &=(\boldsymbol{f}, \boldsymbol{v})\quad  \forall~\boldsymbol{v} \in \boldsymbol{H}_{0}(\grad \curl, \Omega), \\
\label{mixfor2}
c(\boldsymbol{u}, \mu) &=0 \quad\quad\quad  \forall~\mu \in H_{0}^{1}(\Omega),
\end{align}
where the bilinear forms
\[
a(\boldsymbol{u}, \boldsymbol{v})=(\nabla\curl \boldsymbol{u}, \nabla\curl\boldsymbol{v}),\quad  b(\boldsymbol{u}, \boldsymbol{v})=(\curl \boldsymbol{u}, \curl \boldsymbol{v}), \quad c(\boldsymbol{v}, \lambda)=(\boldsymbol{v}, \nabla \lambda).
\]

For $\boldsymbol{v}\in \boldsymbol{H}_{0}(\grad \curl, \Omega)$, equip the squared norm
\[
\|\boldsymbol{v}\|_{\varepsilon}^{2}:=\|\boldsymbol{v}\|_{0}^{2}+\|\curl \boldsymbol{v}\|_{0}^{2}+\varepsilon^{2}|\curl \boldsymbol{v}|_{1}^{2}.
\]
Clearly we have
\[
\varepsilon^{2}a(\boldsymbol{u}, \boldsymbol{v})+b(\boldsymbol{u}, \boldsymbol{v})\leq \|\boldsymbol{u}\|_{\varepsilon}\|\boldsymbol{v}\|_{\varepsilon}\quad\forall~\boldsymbol{u}, \boldsymbol{v}\in \boldsymbol{H}_{0}(\grad \curl, \Omega),
\]
\[
c(\boldsymbol{v}, \mu)\leq \|\boldsymbol{v}\|_{\varepsilon}|\mu|_1\quad\forall~\boldsymbol{v}\in \boldsymbol{H}_{0}(\grad \curl, \Omega), \mu\in H_{0}^{1}(\Omega).
\]

\begin{lemma}
The mixed formulation \eqref{mixfor1}-\eqref{mixfor2} is well-posed, and it holds
\[
\|\boldsymbol{u}\|_{0}+\|\curl \boldsymbol{u}\|_{0}+\varepsilon|\curl \boldsymbol{u}|_{1}\lesssim \|\boldsymbol{f}\|_0.
\]
\end{lemma}
\begin{proof}
For nonzero $\mu\in H_0^1(\Omega)$, it follows the inf-sup condition
\[
|\mu|_1=\frac{c(\nabla\mu, \mu)}{\|\nabla\mu\|_0}=\frac{c(\nabla\mu, \mu)}{\|\nabla\mu\|_{\varepsilon}}\leq\sup_{\boldsymbol{v}\in\boldsymbol{H}_{0}(\grad \curl, \Omega)}\frac{c(\boldsymbol{v}, \mu)}{\|\boldsymbol{v}\|_{\varepsilon}}.
\]
For $\boldsymbol{v}\in\boldsymbol{H}_{0}(\grad \curl, \Omega)$ satisfying $\div\boldsymbol{v}=0$, by Poincar\'e's inequality \cite{ArnoldFalkWinther2006},
\[
\|\boldsymbol{v}\|_{\varepsilon}^2 \lesssim \varepsilon^{2}a(\boldsymbol{v}, \boldsymbol{v})+b(\boldsymbol{v}, \boldsymbol{v}).
\]
Then we end the proof by applying the Babu\v{s}ka-Brezzi theory~\cite{BoffiBrezziFortin2013}.
\end{proof}

By taking $\boldsymbol{v}=\nabla\lambda$ in \eqref{mixfor1}, we get $\lambda = 0$ from the fact $\div \boldsymbol{f} = 0$. Thus \eqref{mixfor1} is reduced to 
\[
\varepsilon^{2}a(\boldsymbol{u}, \boldsymbol{v})+b(\boldsymbol{u}, \boldsymbol{v})=(\boldsymbol{f}, \boldsymbol{v})\quad  \forall~\boldsymbol{v} \in \boldsymbol{H}_{0}(\operatorname{grad} \curl, \Omega).
\]

\subsection{Regularity} 
To derive the regularity of the solution $\boldsymbol{u}$ of the quad-curl singular perturbation problem \eqref{quadcurl}, we introduce the following auxiliary problem, i.e. double curl equation,
\begin{equation}\label{reduced}
\left\{\begin{aligned}
\curl^{2} \boldsymbol{u}_{0}&=\boldsymbol{f} & & \text { in } \Omega, \\
\operatorname{div} \boldsymbol{u}_{0} &=0 & & \text { in } \Omega, \\
\boldsymbol{u}_{0} \times \boldsymbol{n} &=\mathbf{0} & & \text { on } \partial \Omega,
\end{aligned}\right.
\end{equation}
which can be regarded as the limit of problem \eqref{quadcurl} as $\varepsilon$ approaches zero.
Since $\Omega$ is convex, it holds
\begin{equation}\label{eq:regularityu0}
\|\boldsymbol{u}_0\|_1+\|\curl\boldsymbol{u}_0\|_1\lesssim\|\boldsymbol{f}\|_0.
\end{equation}

\begin{lemma}
Let $\boldsymbol{u} \in \boldsymbol{H}_{0}(\grad \curl, \Omega)$ be the solution of problem \eqref{quadcurl} and $\boldsymbol{u}_{0} \in \boldsymbol{H}_{0}(\curl, \Omega)$ be the solution of problem \eqref{reduced}, then we have
\begin{equation}\label{regularity}
\varepsilon^{2}|\curl\boldsymbol{u}|_{2}+\varepsilon|\curl \boldsymbol{u}|_{1}
+\|\boldsymbol{u}-\boldsymbol{u}_{0}\|_{1}\lesssim\varepsilon^{1/2}\|\boldsymbol{f}\|_{0}.
\end{equation}
\end{lemma}
\begin{proof}
Let $\boldsymbol{\phi}=\curl\boldsymbol{u}\in\boldsymbol{H}_0^1(\Omega;\mathbb R^3)$.
It follows from \eqref{eq:curl2laplaceperturb} and \eqref{reduced} that
\[
-\varepsilon^{2} \curl(\Delta\boldsymbol{\phi})=\boldsymbol{f}-\curl^{2} \boldsymbol{u}=\curl^{2}( \boldsymbol{u}_0-\boldsymbol{u}).
\]
Thanks to the de Rham complex, the last equation implies that $\boldsymbol{\phi}\in\boldsymbol{H}_0^1(\Omega;\mathbb R^3)$ satisfies the Stokes equation
\[
\left\{
\begin{aligned}
-\Delta \boldsymbol{\phi} - \nabla p & = \varepsilon^{-2}\curl( \boldsymbol{u}_0-\boldsymbol{u})\quad\;\text { in } \Omega, \\
\div\boldsymbol{\phi}&=0 \qquad\qquad\qquad\qquad\text { in } \Omega.
\end{aligned}
\right.
\]
By the regularity of the Stokes equation \cite[Section 11.5]{MazyaRossmann2010},
\begin{equation}\label{eq:20211223-1}
\|\curl\boldsymbol{u}\|_{2}= \|\boldsymbol{\phi}\|_{2}\lesssim \varepsilon^{-2}\|\curl(\boldsymbol{u}-\boldsymbol{u}_{0})\|_{0}.
\end{equation}

By \eqref{quadcurl} and \eqref{reduced},  we get $\varepsilon^{2} \curl^4\boldsymbol{u}=\curl^{2}( \boldsymbol{u}_0-\boldsymbol{u})$. Then it follows from the integration by parts that
\begin{align*}
\|\curl(\boldsymbol{u}-\boldsymbol{u}_{0})\|_{0}^{2}&= -\varepsilon^{2}(\curl^{3}\boldsymbol{u},\curl( \boldsymbol{u}-\boldsymbol{u}_{0}))\\
&= -\varepsilon^{2}(\curl^{2}\boldsymbol{u}, \curl^{2}(\boldsymbol{u}-\boldsymbol{u}_{0}))
+\varepsilon^{2}(\boldsymbol{n}\times \curl^{2}\boldsymbol{u}, \curl\boldsymbol{u}_{0})_{\partial\Omega}.
\end{align*}
Then we have
\begin{align*}
&\quad\varepsilon^{2}\|\curl^{2}\boldsymbol{u}\|_{0}^{2}+\|\curl(\boldsymbol{u}-\boldsymbol{u}_{0})\|_{0}^{2} \\
& = \varepsilon^{2}(\curl^{2}\boldsymbol{u},\boldsymbol{f})+\varepsilon^{2}(\boldsymbol{n}\times \curl^{2}\boldsymbol{u}, \curl\boldsymbol{u}_{0})_{\partial\Omega}\\
& \leq\varepsilon^{2}\|\curl^{2}\boldsymbol{u}\|_{0}\|\boldsymbol{f}\|_{0}
+\varepsilon^{2}\|\curl^{2}\boldsymbol{u}\|_{0,\partial\Omega}\|\curl\boldsymbol{u}_{0}\|_{0,\partial\Omega}.
\end{align*}
Applying the multiplicative trace inequality, we acquire from \eqref{eq:20211223-1} that
\begin{align*}
&\varepsilon^{2}\|\curl^{2}\boldsymbol{u}\|_{0}^{2}+\|\curl(\boldsymbol{u}-\boldsymbol{u}_{0})\|_{0}^{2}\\
\lesssim& 
\varepsilon^{2}\|\boldsymbol{f}\|_{0}^{2}
+\varepsilon^{2}\|\curl^{2}\boldsymbol{u}\|_{0}^{1/2}\|\curl^{2}\boldsymbol{u}\|_{1}^{1/2}\|\curl\boldsymbol{u}_{0}\|_{1} \\
\lesssim& 
\varepsilon^{2}\|\boldsymbol{f}\|_{0}^{2}
+\varepsilon\|\curl^{2}\boldsymbol{u}\|_{0}^{1/2}\|\curl(\boldsymbol{u}-\boldsymbol{u}_{0})\|_{0}^{1/2}\|\curl\boldsymbol{u}_{0}\|_{1},
\end{align*}
which implies
\[
\varepsilon^{2}\|\curl^{2}\boldsymbol{u}\|_{0}^{2}+\|\curl(\boldsymbol{u}-\boldsymbol{u}_{0})\|_{0}^{2}
\lesssim \varepsilon^{2}\|\boldsymbol{f}\|_{0}^{2}
+\varepsilon\|\curl\boldsymbol{u}_{0}\|_{1}^2.
\]
Employing \eqref{eq:regularityu0} we have
\begin{equation*}
\varepsilon\|\curl^{2}\boldsymbol{u}\|_{0}+\|\curl(\boldsymbol{u}-\boldsymbol{u}_{0})\|_{0}
\lesssim \varepsilon^{1/2}\|\boldsymbol{f}\|_{0}.
\end{equation*}
Since $\curl\boldsymbol{u}\in \boldsymbol{H}_{0}(\div, \Omega)\cap \boldsymbol{H}(\curl, \Omega)$, $\boldsymbol{u}-\boldsymbol{u}_{0}\in \boldsymbol{H}(\div, \Omega)\cap \boldsymbol{H}_{0}(\curl, \Omega)$, and $\div(\boldsymbol{u}-\boldsymbol{u}_{0})=0$, we acquire \cite{GiraultRaviart1986}
\[
\|\curl\boldsymbol{u}\|_{1}\lesssim\|\curl^{2}\boldsymbol{u}\|_{0}\lesssim\varepsilon^{-1/2}\|\boldsymbol{f}\|_{0},
\]
\[
\|\boldsymbol{u}-\boldsymbol{u}_{0}\|_{1}\lesssim\|\curl(\boldsymbol{u}-\boldsymbol{u}_{0})\|_{0}\lesssim\varepsilon^{1/2}\|\boldsymbol{f}\|_{0},
\]
Combining the last two inequalities and \eqref{eq:20211223-1} produces \eqref{regularity}.
\end{proof}

\section{Robust Mixed Finite Element Method}\label{sec:mfem}

We will propose and analyze a robust mixed finite element method for the quad-curl singular perturbation problem in this section.

\subsection{Mixed finite element method}

Based on the mixed formulation \eqref{mixfor1}-\eqref{mixfor2}, a mixed finite element methods for problem \eqref{quadcurl} is to find $\boldsymbol{u}_{h0} \in \boldsymbol{W}_{h0}$ and $\lambda_{h} \in V_{h}^{g}$ such that
\begin{align}
\label{dismixfor1}
\varepsilon^{2}a_{h}(\boldsymbol{u}_{h0}, \boldsymbol{v}_{h})+b(\boldsymbol{u}_{h0}, \boldsymbol{v}_{h})+(\boldsymbol{v}_{h}, \nabla \lambda_{h}) &=(\boldsymbol{f}, \boldsymbol{v}_{h}) \quad \forall~\boldsymbol{v}_{h} \in \boldsymbol{W}_{h0}, \\
\label{dismixfor2}
(\boldsymbol{u}_{h0}, \nabla \mu_{h}) &=0 \quad\quad\quad\;\;  \forall~\mu_h \in V_{h}^{g},
\end{align}
where
\[
a_{h}(\boldsymbol{u}_{h0}, \boldsymbol{v}_{h})=(\nabla_h\curl \boldsymbol{u}_{h0}, \nabla_h\curl\boldsymbol{v}_h) = \sum_{K\in\mathcal T_h}(\nabla\curl \boldsymbol{u}_{h0}, \nabla\curl\boldsymbol{v}_h)_K.
\]

Equip $\boldsymbol{W}_{h}$ with the discrete squared norm
\[
\|\boldsymbol{v}_{h}\|_{\varepsilon,h}^{2}:=\|\boldsymbol{v}_{h}\|_{0}^{2}+\|\curl \boldsymbol{v}_{h}\|_{0}^{2}+\varepsilon^{2}|\curl \boldsymbol{v}_{h}|_{1, h}^{2}.
\]
It is easy to see that
\[
\varepsilon^{2}a_h(\boldsymbol{w}_h, \boldsymbol{v}_h)+b(\boldsymbol{w}_h, \boldsymbol{v}_h)\leq \|\boldsymbol{w}_h\|_{\varepsilon,h}\|\boldsymbol{v}_h\|_{\varepsilon,h}\;\;\forall~\boldsymbol{w}_h, \boldsymbol{v}_h\in \boldsymbol{W}_{h}+\boldsymbol{H}(\grad \curl, \Omega),
\]
\[
c(\boldsymbol{v}_h, \mu_h)\leq \|\boldsymbol{v}_h\|_{\varepsilon,h}|\mu_h|_1\;\;\forall~\boldsymbol{v}_h\in \boldsymbol{W}_{h}+\boldsymbol{H}(\grad \curl, \Omega), \mu_h\in H^1(\Omega).
\]

Next we show the well-posedness of the mixed finite element method \eqref{dismixfor1}-\eqref{dismixfor2} and the stability. 

\begin{lemma}\label{lem:discretepoincare} 
It holds the discrete Poincar\'e inequality
\begin{align}\label{discretepoincare}
\|\boldsymbol{v}_{h}\|_{0} \lesssim\|\curl \boldsymbol{v}_{h}\|_{0} \quad \forall~\boldsymbol{v}_{h} \in \mathcal{K}_{h0}^{d},
\end{align}
where $\mathcal{K}_{h0}^{d}:=\big\{\boldsymbol{v}_{h} \in \boldsymbol{W}_{h0}:(\boldsymbol{v}_{h}, \nabla q_{h})=0$ for each $q_{h} \in V_{h}^{g}\big\}.$
\end{lemma}
\begin{proof}
Noting that $\curl\boldsymbol{v}_{h}\in \boldsymbol{H}_0(\div,\Omega)$, by the de Rham complex, there exists $\boldsymbol{v}\in\boldsymbol{H}_0^1(\Omega;\mathbb R^3)$ such that
\begin{equation}\label{eq:20211223}
\curl\boldsymbol{v}=\curl\boldsymbol{v}_{h},\quad \|\boldsymbol{v}\|_1\lesssim \|\curl\boldsymbol{v}_{h}\|_0.
\end{equation}
Let $\mathcal{I}_{h0}^{gc}\boldsymbol{v}\in\boldsymbol{W}_{h0}$ be the nodal interpolation of $\boldsymbol{v}$ based on DoFs \eqref{Wkdof1}-\eqref{Wkdof3}, and $\mathcal{I}_{h}^{ND}\boldsymbol{v}\in\boldsymbol{V}_{h}^{ND}$ be the nodal interpolation of $\boldsymbol{v}$ based on DoFs \eqref{Wkdof1}-\eqref{Wkdof2}.
Since $\boldsymbol{v}\in\boldsymbol{H}_0^1(\Omega;\mathbb R^3)$ and $\curl\boldsymbol{v}\in\boldsymbol{V}_{h0}^{d}$, both $\mathcal{I}_{h0}^{gc}\boldsymbol{v}$ and $\mathcal{I}_{h}^{ND}\boldsymbol{v}$ are well-defined.
According to (5.43) in \cite{Monk2003}, we have
\[
\|\mathcal{I}_{h}^{ND}\boldsymbol{v}\|_0\lesssim \|\boldsymbol{v}\|_1, \quad \|\curl(\boldsymbol{v}-\mathcal{I}_{h}^{ND}\boldsymbol{v})\|_{0,\partial K}\lesssim h_K^{1/2}|\curl\boldsymbol{v}|_{1,K}\quad\forall~K\in\mathcal T_h.
\]
Employing the scaling argument and the inverse inequality \cite{Ciarlet1978}, it follows
\begin{align*}
\|\mathcal{I}_{h0}^{gc}\boldsymbol{v}-\mathcal{I}_{h}^{ND}\boldsymbol{v}\|_{0,K}&\lesssim h_K^{3/2}\|\curl(\boldsymbol{v}-\mathcal{I}_{h}^{ND}\boldsymbol{v})\|_{0,\partial K} \\
&\lesssim h_K^{2}|\curl\boldsymbol{v}|_{1,K}\lesssim h_K\|\curl\boldsymbol{v}\|_{0,K}.
\end{align*}
Hence
\begin{equation}\label{eq:20211223-3}
\|\mathcal{I}_{h0}^{gc}\boldsymbol{v}\|_0\leq \|\mathcal{I}_{h0}^{gc}\boldsymbol{v}-\mathcal{I}_{h}^{ND}\boldsymbol{v}\|_0+\|\mathcal{I}_{h}^{ND}\boldsymbol{v}\|_0\lesssim \|\boldsymbol{v}\|_1.	
\end{equation}
By the definition of $\mathcal{I}_{h0}^{gc}\boldsymbol{v}$ and the integration by parts, we have
\[
\big(\big(\curl(\mathcal{I}_{h0}^{gc}\boldsymbol{v})-\curl\boldsymbol{v}\big)\cdot\boldsymbol{n}, q\big)_F=0\quad\forall~q \in \mathbb{P}_{1}(F), F\in\mathcal F_h,
\]
\[
\big(\big(\curl(\mathcal{I}_{h0}^{gc}\boldsymbol{v})-\curl\boldsymbol{v}\big)\times\boldsymbol{n}, \boldsymbol{q}\big)_F =0 \quad \forall~\boldsymbol{q}\in \boldsymbol{RT}(F), F\in\mathcal F_h. 
\]
Then we get from the fact $\curl(\mathcal{I}_{h0}^{gc}\boldsymbol{v})-\curl\boldsymbol{v}\in \boldsymbol{V}_{h0}^{d}$ that
\[
\curl(\mathcal{I}_{h0}^{gc}\boldsymbol{v})=\curl\boldsymbol{v}=\curl\boldsymbol{v}_h.
\]
Due to the finite element Stokes complex \eqref{discretecomplex},
$\boldsymbol{v}_h-\mathcal{I}_{h0}^{gc}\boldsymbol{v}\in \boldsymbol{W}_{h0}\cap\ker(\curl)=\nabla V_h^g$. Then we achieve from \eqref{eq:20211223-3} that
\[
\|\boldsymbol{v}_{h}\|_{0}^2=(\boldsymbol{v}_{h},\boldsymbol{v}_{h})=(\boldsymbol{v}_{h},\mathcal{I}_{h0}^{gc}\boldsymbol{v})\leq \|\boldsymbol{v}_{h}\|_0\|\mathcal{I}_{h0}^{gc}\boldsymbol{v}\|_0\lesssim \|\boldsymbol{v}_{h}\|_0\|\boldsymbol{v}\|_1.
\]
Finally \eqref{discretepoincare} follows from the last inequality and \eqref{eq:20211223}.
\end{proof}

\begin{theorem} 
We have the discrete stability
\begin{align}
\label{disstability}
&\|\boldsymbol{\psi}_{h}\|_{\varepsilon,h}+|\gamma_{h}|_{1} \\
\lesssim & \sup _{(\boldsymbol{v}_{h}, \mu_{h}) \in \boldsymbol{W}_{h0} \times V_{h}^{g}}\frac{\varepsilon^{2}a_{h}(\boldsymbol{\psi}_{h}, \boldsymbol{v}_{h})+b(\boldsymbol{\psi}_{h}, \boldsymbol{v}_{h})+(\boldsymbol{v}_{h}, \nabla\gamma_{h})+(\boldsymbol{\psi}_{h}, \nabla \mu_{h})}{\|\boldsymbol{v}_{h}\|_{\varepsilon,h}+\left|\mu_{h}\right|_{1}} \notag
\end{align}  
for any $\boldsymbol{\psi}_{h} \in \boldsymbol{W}_{h0}$ and $\gamma_{h} \in V_{h}^{g}$.
Then the mixed finite element method \eqref{dismixfor1}-\eqref{dismixfor2} is well-posed.
\end{theorem}
\begin{proof}
By \eqref{discretepoincare}, we derive the coercivity
\[
\|\boldsymbol{v}_{h}\|_{\varepsilon,h}^2 \lesssim \varepsilon^{2}a_{h}(\boldsymbol{v}_{h}, \boldsymbol{v}_{h})+b(\boldsymbol{v}_{h}, \boldsymbol{v}_{h})\quad\forall~\boldsymbol{v}_{h}\in\mathcal{K}_{h0}^{d}.
\]
Since $\nabla V_{h}^{g} \subseteq \boldsymbol{W}_{h0}$ by the finite element Stokes complex \eqref{discretecomplex}, it holds the discrete inf-sup condition
\begin{equation*}
|\mu_{h}|_{1}=\sup _{\boldsymbol{v}_{h} \in \nabla V_{h}^{g}} \frac{\left(\boldsymbol{v}_{h}, \nabla \mu_{h}\right)}{\left\|\boldsymbol{v}_{h}\right\|_{0}}
=\sup _{\boldsymbol{v}_{h} \in \nabla V_{h}^{g}} \frac{\left(\boldsymbol{v}_{h}, \nabla \mu_{h}\right)}{\|\boldsymbol{v}_{h}\|_{\varepsilon,h}} 
\leq \sup _{\boldsymbol{v}_{h} \in \boldsymbol{W}_{h0}} \frac{\left(\boldsymbol{v}_{h}, \nabla \mu_{h}\right)}{\|\boldsymbol{v}_{h}\|_{\varepsilon,h}}.
\end{equation*}
Thus the discrete stability \eqref{disstability} follows from the Babu\v{s}ka-Brezzi theory.
\end{proof}

By taking $\boldsymbol{v}_h=\nabla\lambda_h$ in \eqref{dismixfor1}, we get $\lambda_h= 0$. Thus \eqref{dismixfor1} is reduced to 
\begin{align}\label{nolambda}
\varepsilon^{2}a_{h}\left( \boldsymbol{u}_{h0},  \boldsymbol{v}_{h}\right)+b\left( \boldsymbol{u}_{h0},  \boldsymbol{v}_{h}\right)=\left(\boldsymbol{f}, \boldsymbol{v}_{h}\right) \quad \forall~\boldsymbol{v}_{h} \in \boldsymbol{W}_{h0}.
\end{align}

\subsection{Interpolation error estimates}
We will derive some interpolation error estimates for later uses. 
\begin{lemma}
We have
\begin{equation}\label{errorestimateIhd1}
\|\boldsymbol{v}-I_{h}^{d}\boldsymbol{v}\|_{0}+h|\boldsymbol{v}-I_{h}^{d}\boldsymbol{v}|_{1,h}\lesssim h^j|\boldsymbol{v}|_{j}\;\;\;\forall~\boldsymbol{v}\in\boldsymbol{H}_0(\div,\Omega)\cap\boldsymbol{H}^{j}(\Omega; \mathbb{R}^{3}), j=1,2,
\end{equation}
\begin{equation}\label{errorestimateIhd2}
\|I_{h}^{d}\boldsymbol{v}\|_{0}\lesssim \|\boldsymbol{v}\|_{0}+h\|\div\boldsymbol{v}\|_{0}\quad\forall~\boldsymbol{v}\in\boldsymbol{H}_0(\div,\Omega)\cap \boldsymbol{H}^{1}(\Omega; \mathbb{R}^{3}),
\end{equation}
\begin{equation}\label{errorestimateIh0d1}
\|\boldsymbol{v}-I_{h0}^{d}\boldsymbol{v}\|_{0}+h|\boldsymbol{v}-I_{h0}^{d}\boldsymbol{v}|_{1,h}\lesssim h^j|\boldsymbol{v}|_{j}\;\;\;\forall~\boldsymbol{v}\in\boldsymbol{H}_0^{1}(\Omega; \mathbb{R}^{3})\cap\boldsymbol{H}^{j}(\Omega; \mathbb{R}^{3}), j=1,2,
\end{equation}
\begin{equation}\label{errorestimateIh0d2}
\|I_{h0}^{d}\boldsymbol{v}\|_{0}\lesssim \|\boldsymbol{v}\|_{0}+h\|\div\boldsymbol{v}\|_{0}\quad\forall~\boldsymbol{v}\in\boldsymbol{H}_0(\div,\Omega)\cap \boldsymbol{H}^{1}(\Omega; \mathbb{R}^{3}).
\end{equation}
\end{lemma}
\begin{proof}
We only prove \eqref{errorestimateIh0d1} and \eqref{errorestimateIh0d2}, since we can apply the similar argument to prove \eqref{errorestimateIhd1} and \eqref{errorestimateIhd2}.

For $\boldsymbol{v}\in\boldsymbol{H}_0^{1}(\Omega; \mathbb{R}^{3})\cap\boldsymbol{H}^{j}(\Omega; \mathbb{R}^{3})$ and $i=0,1$, we get from the scaling argument that
\[
h_K^i|I_{h}^{BDM}\boldsymbol{v}-I_{h0}^{d}\boldsymbol{v}|_{i,K} \lesssim h_K^{1/2}\|\llbracket I_{h}^{BDM}\boldsymbol{v}\rrbracket\|_{0,\partial K} = h_K^{1/2}\|\llbracket I_{h}^{BDM}\boldsymbol{v}-\boldsymbol{v}\rrbracket\|_{0,\partial K},
\]
which together with the trace inequality yields
\begin{align*}
 & \quad \; |\boldsymbol{v}-I_{h0}^{d}\boldsymbol{v}|_{i,h}^2 \lesssim  |\boldsymbol{v}-I_{h}^{BDM}\boldsymbol{v}|_{i,h}^2 + \sum_{K\in\mathcal T_h}h_K^{1-2i}\|\llbracket I_{h}^{BDM}\boldsymbol{v}-\boldsymbol{v}\rrbracket\|_{0,\partial K}^2 \\
& \lesssim  |\boldsymbol{v}-I_{h}^{BDM}\boldsymbol{v}|_{i,h}^2 + \sum_{K\in\mathcal T_h}h_K^{-2i}(\|\boldsymbol{v}-I_{h}^{BDM}\boldsymbol{v}\|_{0, K}^2+h_K^2|\boldsymbol{v}-I_{h}^{BDM}\boldsymbol{v}|_{1, K}^2).
\end{align*}
Then \eqref{errorestimateIh0d1} holds from \eqref{eq:IhBDMestimate}.

For $\boldsymbol{v}\in\boldsymbol{H}_0(\div,\Omega)\cap \boldsymbol{H}^{1}(\Omega; \mathbb{R}^{3})$, similarly we have
\begin{align*}
\|I_{h}^{BDM}\boldsymbol{v}-I_{h0}^{d}\boldsymbol{v}\|_{0}^2 & \lesssim \sum_{K\in\mathcal T_h}h_K\|\llbracket I_{h}^{BDM}\boldsymbol{v}\rrbracket\|_{0,\partial K}^2 \lesssim \|I_{h}^{BDM}\boldsymbol{v}\|_{0}^2.
\end{align*}
Then
\[
\|I_{h0}^{d}\boldsymbol{v}\|_{0}\leq \|I_{h}^{BDM}\boldsymbol{v}\|_{0}+\|I_{h}^{BDM}\boldsymbol{v}-I_{h0}^{d}\boldsymbol{v}\|_{0}\lesssim \|I_{h}^{BDM}\boldsymbol{v}\|_{0}.
\]
Then \eqref{errorestimateIh0d2} holds from \eqref{eq:IhBDMbound}.
\end{proof}

\begin{lemma}
Assume $\curl\boldsymbol{u}_0\in H^s(\Omega;\mathbb R^3)$ with $1\leq s\leq 2$.
We have
\begin{equation}\label{errorestimateIhd3}
\|\curl\boldsymbol{u}-I_{h}^{d}(\curl\boldsymbol{u})\|_{0}\lesssim \varepsilon^{r-1/2}h^{1-r}\|\boldsymbol{f}\|_0+ h^s|\curl\boldsymbol{u}_0|_{s},
\end{equation}
\begin{equation}\label{errorestimateIhd4}
\varepsilon |\curl\boldsymbol{u}-I_{h}^{d}(\curl\boldsymbol{u})|_{1,h}\lesssim \varepsilon^{r-1/2}h^{1-r}\|\boldsymbol{f}\|_0
\end{equation}
with $0\leq r\leq 1$.
\end{lemma}
\begin{proof}
It follows from \eqref{errorestimateIhd1}-\eqref{errorestimateIhd2} and \eqref{eq:regularityu0}-\eqref{regularity} that
\begin{align*}
\|\curl(\boldsymbol{u}-\boldsymbol{u}_0)-I_{h}^{d}(\curl(\boldsymbol{u}-\boldsymbol{u}_0))\|_{0}&\lesssim h^{1-r}\|\curl(\boldsymbol{u}-\boldsymbol{u}_0)\|_0^r|\curl(\boldsymbol{u}-\boldsymbol{u}_0)|_1^{1-r} \\
&\lesssim \varepsilon^{r-1/2}h^{1-r}\|\boldsymbol{f}\|_0.
\end{align*}
Apply \eqref{errorestimateIhd1} to get
\[
\|\curl\boldsymbol{u}_0-I_{h}^{d}(\curl\boldsymbol{u}_0)\|_{0}\lesssim h^s|\curl\boldsymbol{u}_0|_{s}.
\]
Hence \eqref{errorestimateIhd3} holds from the last two inequalities.

Employing \eqref{errorestimateIhd1} again, we get 
\[
|\curl\boldsymbol{u}-I_{h}^{d}(\curl\boldsymbol{u})|_{1,h}\lesssim h^{1-r}|\curl\boldsymbol{u}|_1^r|\curl\boldsymbol{u}|_2^{1-r},
\]
which combined with \eqref{regularity} produces \eqref{errorestimateIhd4}.
\end{proof}

\begin{lemma}
We have
\begin{equation}\label{errorestimateIh0d3}
\|\curl\boldsymbol{u}-I_{h0}^{d}(\curl\boldsymbol{u})\|_{0}+\varepsilon |\curl\boldsymbol{u}-I_{h0}^{d}(\curl\boldsymbol{u})|_{1,h}\lesssim h^{1/2}\|\boldsymbol{f}\|_0.
\end{equation}
\end{lemma}
\begin{proof}
With the help of the scaling argument, 
\[
\|I_{h}^{d}(\curl\boldsymbol{u})-I_{h0}^{d}(\curl\boldsymbol{u})\|_{0}^2\lesssim \sum_{F\in\mathcal F_h^{\partial}}h_F\|\curl\boldsymbol{u}-I_{h}^{BDM}(\curl\boldsymbol{u})\|_{0,F}^2.
\]
On the other hand,
by the multiplicative trace inequality, \eqref{eq:IhBDMestimate} and \eqref{eq:regularityu0}-\eqref{regularity},
we acquire
\begin{align*}
&\quad \sum_{F\in\mathcal F_h^{\partial}}h_F\|\curl(\boldsymbol{u}-\boldsymbol{u}_0)-I_{h}^{BDM}(\curl(\boldsymbol{u}-\boldsymbol{u}_0))\|_{0,F}^2 \\
&\lesssim h\|\curl(\boldsymbol{u}-\boldsymbol{u}_0)\|_{0}|\curl(\boldsymbol{u}-\boldsymbol{u}_0)|_{1}\lesssim h\|\boldsymbol{f}\|_0^2,
\end{align*}
\[
\sum_{F\in\mathcal F_h^{\partial}}h_F\|\curl\boldsymbol{u}_0-I_{h}^{BDM}(\curl\boldsymbol{u}_0)\|_{0,F}^2\lesssim h^2|\curl\boldsymbol{u}_0|_{1}^2\lesssim h^2\|\boldsymbol{f}\|_0^2.
\]
Hence
\[
\|I_{h}^{d}(\curl\boldsymbol{u})-I_{h0}^{d}(\curl\boldsymbol{u})\|_{0}\lesssim h^{1/2}\|\boldsymbol{f}\|_0.
\]
Combining the last inequality and \eqref{errorestimateIhd3} with $r=1/2$ and $s=1$ we derive
\begin{equation}\label{errorestimateIh0d30}
\|\curl\boldsymbol{u}-I_{h0}^{d}(\curl\boldsymbol{u})\|_{0}\lesssim h^{1/2}\|\boldsymbol{f}\|_0.
\end{equation}

Applying the scaling argument, the trace inequality, \eqref{eq:IhBDMestimate} and \eqref{regularity} again, 
\begin{align*}
|I_{h}^{d}(\curl\boldsymbol{u})-I_{h0}^{d}(\curl\boldsymbol{u})|_{1,h}^2&\lesssim \sum_{F\in\mathcal{F}_h^{\partial}}h_F^{-1}\|\curl\boldsymbol{u}-I_{h}^{BDM}(\curl\boldsymbol{u})\|_{0,F}^2 \\
&\lesssim h|\curl\boldsymbol{u}|_1|\curl\boldsymbol{u}|_2\lesssim \varepsilon^{-2}h\|\boldsymbol{f}\|_0^2
\end{align*}
Then we get from \eqref{errorestimateIhd4} with $r=1/2$ that
\[
\varepsilon|\curl\boldsymbol{u}-I_{h0}^{d}(\curl\boldsymbol{u})|_{1,h}\lesssim h^{1/2}\|\boldsymbol{f}\|_0,
\]
which together with \eqref{errorestimateIh0d30} implies \eqref{errorestimateIh0d3}.
\end{proof}

\begin{lemma}
The following estimates hold 
\begin{equation}\label{errorestimateIhgc1}
\|\boldsymbol{v}-I_{h}^{gc}\boldsymbol{v}\|_{0}\lesssim h^j|\boldsymbol{v}|_{j}\quad\forall~\boldsymbol{v}\in\boldsymbol{H}_0(\curl,\Omega)\cap\boldsymbol{H}^{j}(\Omega; \mathbb{R}^{3}), j=2,k+1,
\end{equation}
\begin{equation}\label{errorestimateIhgc2}
\|\boldsymbol{v}-I_{h}^{gc}\boldsymbol{v}\|_{0}\lesssim h|\boldsymbol{v}|_{1}\quad\forall~\boldsymbol{v}\in\boldsymbol{H}_0(\curl,\Omega)\cap\boldsymbol{H}^{1}(\Omega; \mathbb{R}^{3}).
\end{equation}
\begin{equation}\label{errorestimateIh0gc1}
\|\boldsymbol{v}-I_{h0}^{gc}\boldsymbol{v}\|_{0}\lesssim h^j|\boldsymbol{v}|_{j}\quad\forall~\boldsymbol{v}\in\boldsymbol{H}_0(\grad\curl, \Omega)\cap\boldsymbol{H}^{j}(\Omega; \mathbb{R}^{3}), j=2,k+1,
\end{equation}
\begin{equation}\label{errorestimateIh0gc2}
\|\boldsymbol{v}-I_{h0}^{gc}\boldsymbol{v}\|_{0}\lesssim h|\boldsymbol{v}|_{1}\quad\forall~\boldsymbol{v}\in\boldsymbol{H}_0(\curl,\Omega)\cap\boldsymbol{H}^{1}(\Omega; \mathbb{R}^{3}).
\end{equation}
\end{lemma}
\begin{proof}
We only prove \eqref{errorestimateIh0gc1} and \eqref{errorestimateIh0gc2}, since we can apply the similar argument to prove \eqref{errorestimateIhgc1} and \eqref{errorestimateIhgc2}.

For $K\in\mathcal T_h$, it follows from the scaling argument, commutative diagram \eqref{eq:feec}, commutative diagram \eqref{eq:stokescdncfem0}, and the inverse inequality that
\begin{align*}
&\quad\|I_{h}^{ND}\boldsymbol{v}-I_{h0}^{gc}\boldsymbol{v}\|_{0,K} \\
&\lesssim h_K^{3/2}\|\curl(I_{h}^{ND}\boldsymbol{v}-I_{h0}^{gc}\boldsymbol{v})\|_{0,\partial K}= h_K^{3/2}\|I_{h}^{BDM}(\curl\boldsymbol{v})-I_{h0}^{d}(\curl\boldsymbol{v})\|_{0,\partial K} \\
&\lesssim h_K\|I_{h}^{BDM}(\curl\boldsymbol{v})-I_{h0}^{d}(\curl\boldsymbol{v})\|_{0,K}.
\end{align*}
Hence
\[
\|\boldsymbol{v}-I_{h0}^{gc}\boldsymbol{v}\|_{0,K}\lesssim \|\boldsymbol{v}-I_{h}^{ND}\boldsymbol{v}\|_{0,K} + h_K\|I_{h}^{BDM}(\curl\boldsymbol{v})-I_{h0}^{d}(\curl\boldsymbol{v})\|_{0,K}.
\]
Thus estimates \eqref{errorestimateIh0gc1}-\eqref{errorestimateIh0gc2} hold from \eqref{eq:IhNDestimate}-\eqref{eq:IhBDMestimate} and \eqref{errorestimateIh0d1}-\eqref{errorestimateIh0d2}.
\end{proof}


\subsection{Error analysis}

We first present the consistency error of the mixed finite element method \eqref{dismixfor1}-\eqref{dismixfor2}.

\begin{lemma}\label{lem:consistencyerror}
We have for any $\boldsymbol{v}_{h} \in \boldsymbol{W}_{h0}$ that
\begin{equation}
\label{consistencyerror1}
\varepsilon^{2}a_{h}(\boldsymbol{u}, \boldsymbol{v}_{h})+b(\boldsymbol{u}, \boldsymbol{v}_{h})-(\boldsymbol{f}, \boldsymbol{v}_{h}) \lesssim h\varepsilon^{2}|\curl \boldsymbol{u}|_{2}\left|\curl \boldsymbol{v}_{h}\right|_{1, h},
\end{equation}
\begin{equation}
\label{consistencyerror2}
\varepsilon^{2}a_{h}(\boldsymbol{u}, \boldsymbol{v}_{h})+b(\boldsymbol{u}, \boldsymbol{v}_{h})-(\boldsymbol{f}, \boldsymbol{v}_{h}) \lesssim h^{1/2}\varepsilon^{2}|\curl\boldsymbol{u}|_{1}^{1/2}|\curl\boldsymbol{u}|_{2}^{1/2}|\curl\boldsymbol{v}_{h}|_{1, h}.
\end{equation}
\end{lemma}
\begin{proof}
Since $\boldsymbol{v}_{h}\in\boldsymbol{H}_{0}(\curl, \Omega)$,
we derive from \eqref{eq:curl2laplaceperturb} and the integration by parts that
\[
-\varepsilon^{2}(\Delta\curl\boldsymbol{u}, \curl\boldsymbol{v}_h)+ b(\boldsymbol{u}, \boldsymbol{v}_h)=(\boldsymbol{f}, \boldsymbol{v}_h).
\]
Then
\begin{align}
\varepsilon^{2}a_{h}\left(\boldsymbol{u}, \boldsymbol{v}_{h}\right)+b\left(\boldsymbol{u}, \boldsymbol{v}_{h}\right)-\left(\boldsymbol{f}, \boldsymbol{v}_{h}\right) &=\varepsilon^{2}a_{h}\left(\boldsymbol{u}, \boldsymbol{v}_{h}\right) + \varepsilon^{2}(\Delta\curl\boldsymbol{u}, \curl\boldsymbol{v}_h) \notag\\
&=\sum_{K \in \mathcal{T}_{h}}\varepsilon^{2}\left(\partial_{n}(\curl \boldsymbol{u}), \curl \boldsymbol{v}_{h}\right)_{\partial K}. \label{eq:202112124}
\end{align}
On the other hand, due to the weak continuity of $\curl \boldsymbol{v}_{h}\in\boldsymbol{V}_{h0}^d$, we get from the trace inequality and the error estimate of $Q_{F}^{0}$ that
\begin{align*}
&\quad \sum_{K \in \mathcal{T}_{h}}\varepsilon^{2}\left(\partial_{n}(\curl \boldsymbol{u}), \curl \boldsymbol{v}_{h}\right)_{\partial K} \\
&= \sum_{K \in \mathcal{T}_{h}} \sum_{F \in \mathcal{F}(K)}\varepsilon^{2}\left(\partial_{n}(\curl \boldsymbol{u})-Q_{F}^{0} \partial_{n}(\curl \boldsymbol{u}), \curl \boldsymbol{v}_{h}\right)_{F} \\
&= \sum_{K \in \mathcal{T}_{h}} \sum_{F \in \mathcal{F}(K)}\varepsilon^{2}\left(\partial_{n}(\curl \boldsymbol{u})-Q_{F}^{0} \partial_{n}(\curl \boldsymbol{u}), \curl \boldsymbol{v}_{h}-Q_{F}^{0} \curl \boldsymbol{v}_{h}\right)_{F} \\
&\lesssim \min\big\{h\varepsilon^{2}|\curl \boldsymbol{u}|_{2}|\curl \boldsymbol{v}_{h}|_{1, h}, h^{1/2}\varepsilon^{2}|\curl\boldsymbol{u}|_{1}^{1/2}|\curl\boldsymbol{u}|_{2}^{1/2}|\curl\boldsymbol{v}_{h}|_{1, h}\big\}.
\end{align*}
Thus \eqref{consistencyerror1} and \eqref{consistencyerror2} follow from \eqref{eq:202112124}.
\end{proof}

Now we can show the a priori error estimate.

\begin{theorem}\label{thm:priorierror}
Let $\boldsymbol{u} \in \boldsymbol{H}_{0}(\grad \curl, \Omega)$ be the solution of the problem \eqref{quadcurl}, and $\boldsymbol{u}_{h0} \in \boldsymbol{W}_{h0}$ be the solution of the mixed finite element method \eqref{dismixfor1}-\eqref{dismixfor2}. Then
\begin{align}
\label{priorierror}
\|\boldsymbol{u}-\boldsymbol{u}_{h0}\|_{\varepsilon,h} &\lesssim \inf\limits_{\boldsymbol{w}_{h} \in \boldsymbol{W}_{h0}}\|\boldsymbol{u}-\boldsymbol{w}_h\|_{\varepsilon, h} \\
&\quad\; +\sup\limits_{\boldsymbol{v}_{h} \in \boldsymbol{W}_{h0}} \frac{\varepsilon^{2}a_{h}(\boldsymbol{u}, \boldsymbol{v}_{h})+b(\boldsymbol{u}, \boldsymbol{v}_{h})-(\boldsymbol{f}, \boldsymbol{v}_{h})}{\|\boldsymbol{v}_{h}\|_{\varepsilon, h}}. \notag
\end{align}
\end{theorem}
\begin{proof}
For $\boldsymbol{w}_{h}, \boldsymbol{v}_{h} \in \boldsymbol{W}_{h0}$,
we get from \eqref{nolambda} and the fact $\boldsymbol{u}_{h0} \in \mathcal{K}_{h0}^{d}$ that
\begin{align*}
&\quad \varepsilon^{2}a_{h}(\boldsymbol{w}_{h}-\boldsymbol{u}_{h0},  \boldsymbol{v}_{h})+b(\boldsymbol{w}_{h}-\boldsymbol{u}_{h0},  \boldsymbol{v}_{h}) + c(\boldsymbol{w}_{h}-\boldsymbol{u}_{h0}, \mu_{h}) \\
&=\varepsilon^{2}a_{h}(\boldsymbol{w}_{h},  \boldsymbol{v}_{h})+b(\boldsymbol{w}_{h},  \boldsymbol{v}_{h}) + c(\boldsymbol{w}_{h}, \mu_{h}) -(\boldsymbol{f}, \boldsymbol{v}_{h}) \\
&=\varepsilon^{2}a_{h}(\boldsymbol{w}_{h}-\boldsymbol{u},  \boldsymbol{v}_{h})+b(\boldsymbol{w}_{h}-\boldsymbol{u},  \boldsymbol{v}_{h}) + c(\boldsymbol{w}_{h}-\boldsymbol{u}, \mu_{h}) \\
&\quad + \varepsilon^{2}a_{h}(\boldsymbol{u}, \boldsymbol{v}_{h})+b(\boldsymbol{u}, \boldsymbol{v}_{h})-(\boldsymbol{f}, \boldsymbol{v}_{h}) \\
&\lesssim \|\boldsymbol{u}-\boldsymbol{w}_{h}\|_{\varepsilon, h}(\|\boldsymbol{v}_h\|_{\varepsilon, h}+|\mu_h|_1) + \varepsilon^{2}a_{h}(\boldsymbol{u}, \boldsymbol{v}_{h})+b(\boldsymbol{u}, \boldsymbol{v}_{h})-(\boldsymbol{f}, \boldsymbol{v}_{h}).
\end{align*}
By the discrete stability \eqref{disstability} with $\boldsymbol{\psi}_{h} = \boldsymbol{w}_{h}-\boldsymbol{u}_{h0}$ and
$\gamma_{h} = 0$, we get
\begin{align*}
&\|\boldsymbol{w}_{h}-\boldsymbol{u}_{h0}\|_{\varepsilon,h}\\
\lesssim & \sup _{\left(\boldsymbol{v}_{h}, \mu_{h}\right) \in \boldsymbol{W}_{h0} \times V_{h}^{g}} \frac{\varepsilon^{2}a_{h}(\boldsymbol{w}_{h}-\boldsymbol{u}_{h0},  \boldsymbol{v}_{h})+b(\boldsymbol{w}_{h}-\boldsymbol{u}_{h0},  \boldsymbol{v}_{h})+(\boldsymbol{w}_{h}-\boldsymbol{u}_{h0}, \nabla \mu_{h})}{\|\boldsymbol{v}_{h}\|_{\varepsilon, h}+|\mu_{h}|_{1}} \\
\lesssim &\|\boldsymbol{u}-\boldsymbol{w}_{h}\|_{\varepsilon, h}+\sup _{\boldsymbol{v}_{h} \in \boldsymbol{W}_{h0}} \frac{\varepsilon^{2}a_{h}(\boldsymbol{u}, \boldsymbol{v}_{h})+b(\boldsymbol{u}, \boldsymbol{v}_{h})-(\boldsymbol{f}, \boldsymbol{v}_{h})}{\|\boldsymbol{v}_{h}\|_{\varepsilon, h}},
\end{align*}
which together with the triangle inequality implies \eqref{priorierror}.
\end{proof}

\begin{corollary}\label{cor:priorierrornotrobust}
Let $\boldsymbol{u} \in \boldsymbol{H}_{0}(\grad \curl, \Omega)$ be the solution of the problem \eqref{quadcurl}, and $\boldsymbol{u}_{h0} \in \boldsymbol{W}_{h0}$ be the solution of the mixed finite element method \eqref{dismixfor1}-\eqref{dismixfor2}. Assume $\boldsymbol{u} \in \boldsymbol{H}^{k+1}(\Omega; \mathbb{R}^{3})$ and $\curl\boldsymbol{u} \in \boldsymbol{H}^{2}(\Omega; \mathbb{R}^{3})$. Then we have
\begin{equation}
\label{priorierrornotrobust}
\|\boldsymbol{u}-\boldsymbol{u}_{h0}\|_{\varepsilon,h} \lesssim h^{k+1}|\boldsymbol{u}|_{k+1}+ (h^2+h\varepsilon)|\curl\boldsymbol{u}|_{2}.
\end{equation}
\end{corollary}
\begin{proof}
It follows from \eqref{errorestimateIh0gc1} that
\[
\|\boldsymbol{u}-I_{h0}^{gc}\boldsymbol{u}\|_{0}\lesssim h^{k+1}|\boldsymbol{u}|_{k+1}.
\]
By commutative diagram \eqref{eq:stokescdncfem0} and \eqref{errorestimateIh0d1},
\[
h^j|\curl(\boldsymbol{u}-I_{h0}^{gc}\boldsymbol{u})|_{j,h}=h^j|\curl\boldsymbol{u}-I_{h0}^{d}(\curl\boldsymbol{u})|_{j,h}\lesssim h^{2}|\curl\boldsymbol{u}|_{2},\; j=0,1.
\]
Therefore \eqref{priorierrornotrobust} holds from \eqref{consistencyerror1} and \eqref{priorierror} with $\boldsymbol{w}_h=I_{h0}^{gc}\boldsymbol{u}$.
\end{proof}

Corollary~\ref{cor:priorierrornotrobust} ensures the linear convergence $O(h)$ of $\|\boldsymbol{u}-\boldsymbol{u}_{h0}\|_{\varepsilon,h} $ for any fixed $\varepsilon>0$. However, by the regularity \eqref{regularity}, $|\curl\boldsymbol{u}|_2=O(\varepsilon^{-3/2})$ depends on parameter $\varepsilon$, therefore the upper bound $h^{k+1}|\boldsymbol{u}|_{k+1}+ (h^2+h\varepsilon)|\curl\boldsymbol{u}|_{2}$ in \eqref{priorierrornotrobust} will deteriorate when $\varepsilon$ approaches zero. In the following we show a uniform estimate of $\|\boldsymbol{u}-\boldsymbol{u}_{h0}\|_{\varepsilon,h}$ with respect to $\varepsilon$.


\begin{theorem}
Let $\boldsymbol{u} \in \boldsymbol{H}_{0}(\grad \curl, \Omega)$ be the solution of the problem \eqref{quadcurl}, $\boldsymbol{u}_{0} \in \boldsymbol{H}_{0}(\curl, \Omega)$ be the solution of double curl equation \eqref{reduced}, and $\boldsymbol{u}_{h0} \in \boldsymbol{W}_{h0}$ be the solution of the mixed finite element methods \eqref{dismixfor1}-\eqref{dismixfor2}. Then we have the robust error estimates
\begin{equation}
\label{priorierrorrobust}
\|\boldsymbol{u}-\boldsymbol{u}_{h0}\|_{\varepsilon,h} \lesssim h^{1/2}\|\boldsymbol{f}\|_{0},
\end{equation}
\begin{equation}
\label{priorierrorrobustu0}
\|\boldsymbol{u}_0-\boldsymbol{u}_{h0}\|_{\varepsilon,h} \lesssim (\varepsilon^{1/2}+h^{1/2})\|\boldsymbol{f}\|_{0}.
\end{equation}
\end{theorem}
\begin{proof}
It follows from \eqref{errorestimateIh0gc2} and \eqref{eq:regularityu0}-\eqref{regularity} that
\[
\|\boldsymbol{u}-I_{h0}^{gc}\boldsymbol{u}\|_{0}\lesssim h|\boldsymbol{u}|_1\leq h|\boldsymbol{u}-\boldsymbol{u}_0|_1+h|\boldsymbol{u}_0|_1\lesssim h\|\boldsymbol{f}\|_{0}.
\]
Then we get from \eqref{errorestimateIh0d3} and $\curl(I_{h0}^{gc}\boldsymbol{u})=I_{h0}^{d}(\curl\boldsymbol{u})$ that
\[
\|\boldsymbol{u}-I_{h0}^{gc}\boldsymbol{u}\|_{\varepsilon,h} \lesssim h^{1/2}\|\boldsymbol{f}\|_{0}.
\]
On the other side, it holds from \eqref{consistencyerror2} and \eqref{regularity} that
\[
\sup\limits_{\boldsymbol{v}_{h} \in \boldsymbol{W}_{h0}} \frac{\varepsilon^{2}a_{h}(\boldsymbol{u}, \boldsymbol{v}_{h})+b(\boldsymbol{u}, \boldsymbol{v}_{h})-(\boldsymbol{f}, \boldsymbol{v}_{h})}{\|\boldsymbol{v}_{h}\|_{\varepsilon, h}}\lesssim h^{1/2}\varepsilon|\curl\boldsymbol{u}|_{1}^{1/2}|\curl\boldsymbol{u}|_{2}^{1/2}\lesssim h^{1/2}\|\boldsymbol{f}\|_{0}.
\]
Hence we conclude \eqref{priorierrorrobust} from \eqref{priorierror} and the last two inequalities.
Finally combining \eqref{priorierrorrobust} and regularities \eqref{eq:regularityu0}-\eqref{regularity} produces \eqref{priorierrorrobustu0}.
\end{proof}

\section{Modified Mixed Finite Element Method Based on Nitsche's Technique}\label{sec:modifiedmfem}

The half-order convergence rate of $\|\boldsymbol{u}-\boldsymbol{u}_{h}\|_{\varepsilon,h}$ in \eqref{priorierrorrobust} is sharp, but not optimal. To this end, we will develop a modified mixed finite element method based on the Nitsche's technique following the ideas in \cite{Nitsche1971,Schieweck2008,GuzmanLeykekhmanNeilan2012} in this section, which possesses the optimal convergence rate for small $\varepsilon$. 

\subsection{Modified mixed finite element method}
By applying the Nitsche's technique to the linear form $a_{h}(\boldsymbol{u}_{h}, \boldsymbol{v}_{h})$ in the mixed finite element method \eqref{dismixfor1}-\eqref{dismixfor2}, the modified mixed finite element method is to find $\boldsymbol{u}_{h} \in \boldsymbol{W}_{h}$ and $\lambda_{h} \in V_{h}^{g}$ such that
\begin{align}
\label{dismixforNitsche1}
\varepsilon^{2}\tilde a_{h}(\boldsymbol{u}_{h}, \boldsymbol{v}_{h})+b(\boldsymbol{u}_{h}, \boldsymbol{v}_{h})+(\boldsymbol{v}_{h}, \nabla \lambda_{h}) &=(\boldsymbol{f}, \boldsymbol{v}_{h}) \quad \forall~\boldsymbol{v}_{h} \in \boldsymbol{W}_{h}, \\
\label{dismixforNitsche2}
(\tilde{\boldsymbol{u}}_{h}, \nabla \mu_{h}) &=0 \quad\quad\quad\;\;  \forall~\mu_h \in V_{h}^{g},
\end{align}
where
\begin{align*}
\tilde a_{h}(\boldsymbol{u}_{h}, \boldsymbol{v}_{h})&=(\nabla_h\curl\boldsymbol{u}_h, \nabla_h\curl\boldsymbol{v}_h) - \sum_{F\in\mathcal F_h^{\partial}}\big(\partial_n(\curl\boldsymbol{u}_{h}), \curl\boldsymbol{v}_{h}\big)_F \\
&\quad - \sum_{F\in\mathcal F_h^{\partial}}\big(\curl\boldsymbol{u}_{h}, \partial_n(\curl\boldsymbol{v}_{h})\big)_F+ \sum_{F\in\mathcal F_h^{\partial}}\frac{\sigma}{h_F}\big(\curl\boldsymbol{u}_{h}, \curl\boldsymbol{v}_{h}\big)_F
\end{align*}
with constant $\sigma>0$.

Equip $\boldsymbol{W}_{h}$ with the discrete squared norm
\begin{align*}
\interleave\boldsymbol{v}_{h}\interleave_{\varepsilon,h}^{2}&:=\|\boldsymbol{v}_{h}\|_{\varepsilon,h}^{2}+\varepsilon^{2}\sum_{F\in\mathcal F_h^{\partial}}\frac{1}{h_F}\|\curl\boldsymbol{v}_{h}\|_F^2\\
&\;=\|\boldsymbol{v}_{h}\|_{0}^{2}+\|\curl \boldsymbol{v}_{h}\|_{0}^{2}+\varepsilon^{2}|\curl \boldsymbol{v}_{h}|_{1, h}^{2}+\varepsilon^{2}\sum_{F\in\mathcal F_h^{\partial}}\frac{1}{h_F}\|\curl\boldsymbol{v}_{h}\|_F^2.
\end{align*}
There exists a constant $\sigma_0>0$ depending only on the shape regularity of $\mathcal T_h$ such that for any fixed number $\sigma\geq\sigma_0$, it holds \cite[Lemma 2]{Wheeler1978}
\begin{equation*}
|\curl \boldsymbol{v}_{h}|_{1, h}^{2}+\sum_{F\in\mathcal F_h^{\partial}}\frac{1}{h_F}\|\curl\boldsymbol{v}_{h}\|_F^2\lesssim \tilde a_{h}(\boldsymbol{v}_{h}, \boldsymbol{v}_{h}) \quad\forall~\boldsymbol{v}_h\in \boldsymbol{W}_{h}.
\end{equation*}

Adopting the similar argument as in Section~\ref{sec:mfem},
we have the discrete stability
\begin{align*}
&\|\boldsymbol{\psi}_{h}\|_{\varepsilon,h}+|\gamma_{h}|_{1} \\
\lesssim & \sup _{(\boldsymbol{v}_{h}, \mu_{h}) \in \boldsymbol{W}_{h} \times V_{h}^{g}}\frac{\varepsilon^{2}\tilde a_{h}(\boldsymbol{\psi}_{h}, \boldsymbol{v}_{h})+b(\boldsymbol{\psi}_{h}, \boldsymbol{v}_{h})+(\boldsymbol{v}_{h}, \nabla\gamma_{h})+(\boldsymbol{\psi}_{h}, \nabla \mu_{h})}{\interleave\boldsymbol{v}_{h}\interleave_{\varepsilon,h}+|\mu_{h}|_{1}} 
\end{align*}  
for any $\boldsymbol{\psi}_{h} \in \boldsymbol{W}_{h}$ and $\gamma_{h} \in V_{h}^{g}$.
Therefore the modified mixed finite element method \eqref{dismixforNitsche1}-\eqref{dismixforNitsche2} is well-posed.

\subsection{Error analysis}
First apply the same argument as the proof of Theorem~\ref{thm:priorierror} and Corollary~\ref{cor:priorierrornotrobust} to derive the following error estimates.
\begin{theorem}
Let $\boldsymbol{u} \in \boldsymbol{H}_{0}(\grad \curl, \Omega)$ be the solution of the problem \eqref{quadcurl}, and $\boldsymbol{u}_{h} \in\boldsymbol{W}_{h}$ be the solution of the mixed finite element method \eqref{dismixforNitsche1}-\eqref{dismixforNitsche2}. Then
\begin{align}
\label{priorierrorNitsche}
\interleave\boldsymbol{u}-\boldsymbol{u}_{h}\interleave_{\varepsilon,h} &\lesssim \inf\limits_{\boldsymbol{w}_{h} \in \boldsymbol{W}_{h}}\interleave\boldsymbol{u}-\boldsymbol{w}_h\interleave_{\varepsilon, h} \\
&\quad\; +\sup\limits_{\boldsymbol{v}_{h} \in \boldsymbol{W}_{h}} \frac{\varepsilon^{2}\tilde a_{h}(\boldsymbol{u}, \boldsymbol{v}_{h})+b(\boldsymbol{u}, \boldsymbol{v}_{h})-(\boldsymbol{f}, \boldsymbol{v}_{h})}{\interleave\boldsymbol{v}_{h}\interleave_{\varepsilon, h}}. \notag
\end{align}
Assume $\boldsymbol{u} \in \boldsymbol{H}^{k+1}(\Omega; \mathbb{R}^{3})$ and $\curl\boldsymbol{u} \in \boldsymbol{H}^{2}(\Omega; \mathbb{R}^{3})$. Then we have
\begin{equation*}
\interleave\boldsymbol{u}-\boldsymbol{u}_{h}\interleave_{\varepsilon,h} \lesssim h^{k+1}|\boldsymbol{u}|_{k+1}+ (h^2+h\varepsilon)|\curl\boldsymbol{u}|_{2}.
\end{equation*}
\end{theorem}

Finally we derive a optimal and robust error estimate of  $\interleave\boldsymbol{u}-\boldsymbol{u}_{h}\interleave_{\varepsilon,h}$.

\begin{lemma} 
We have for any $\boldsymbol{v}_{h} \in \boldsymbol{W}_{h}$ that
\begin{equation}
\label{consistencyerror3}
\varepsilon^{2}\tilde a_{h}(\boldsymbol{u}, \boldsymbol{v}_{h})+b(\boldsymbol{u}, \boldsymbol{v}_{h})-(\boldsymbol{f}, \boldsymbol{v}_{h}) \lesssim \varepsilon^{r-1/2}h^{1-r}\|\boldsymbol{f}\|_0\interleave \boldsymbol{v}_h\interleave_{\varepsilon,h}
\end{equation}
with $0\leq r\leq 1$.
\end{lemma}
\begin{proof}
Applying the similar argument as in the proof of Lemma~\ref{lem:consistencyerror}, we have
\begin{align*}
&\quad \;\varepsilon^{2}\tilde a_{h}(\boldsymbol{u}, \boldsymbol{v}_{h})+b(\boldsymbol{u}, \boldsymbol{v}_{h})-(\boldsymbol{f}, \boldsymbol{v}_{h}) \\
&= \sum_{K \in \mathcal{T}_{h}} \sum_{F \in \mathcal{F}(K)\cap\mathcal{F}_h^i}\varepsilon^{2}\left(\partial_{n}(\curl \boldsymbol{u})-Q_{F}^{0} \partial_{n}(\curl \boldsymbol{u}), \curl \boldsymbol{v}_{h}-Q_{F}^{0} \curl \boldsymbol{v}_{h}\right)_{F}.
\end{align*}
By the error estimate of $Q_{F}^{0}$ and the inverse inequality,
\begin{align*}
\varepsilon^{2}\tilde a_{h}(\boldsymbol{u}, \boldsymbol{v}_{h})+b(\boldsymbol{u}, \boldsymbol{v}_{h})-(\boldsymbol{f}, \boldsymbol{v}_{h}) &\lesssim \varepsilon^{2}h^{1-r}|\curl\boldsymbol{u}|_2\|\curl\boldsymbol{v}_h\|_0^r|\curl\boldsymbol{v}_h|_{1,h}^{1-r}\\
&\lesssim \varepsilon^{r+1}h^{1-r}|\curl\boldsymbol{u}|_2\interleave \boldsymbol{v}_h\interleave_{\varepsilon,h},
\end{align*}
which together with \eqref{regularity} indicates \eqref{consistencyerror3}.
\end{proof}

\begin{lemma}
Assume $\boldsymbol{u}_0\in H^{s+1}(\Omega;\mathbb R^3)$ with $1\leq s\leq k$.
We have
\begin{equation}\label{errorestimateIhd5}
\interleave\boldsymbol{u}-I_{h}^{gc}\boldsymbol{u}\interleave_{\varepsilon,h}\lesssim \varepsilon^{r-1/2}h^{1-r}\|\boldsymbol{f}\|_0+ h^s(h|\boldsymbol{u}_0|_{s+1}+|\curl\boldsymbol{u}_0|_s)
\end{equation}
with $0\leq r\leq 1$.
\end{lemma}
\begin{proof}
By \eqref{errorestimateIhgc1}-\eqref{errorestimateIhgc2} and \eqref{regularity},
\begin{align*}
\|\boldsymbol{u}-I_{h}^{gc}\boldsymbol{u}\|_0&\leq \|(\boldsymbol{u}-\boldsymbol{u}_0)-I_{h}^{gc}(\boldsymbol{u}-\boldsymbol{u}_0)\|_0+\|\boldsymbol{u}_0-I_{h}^{gc}\boldsymbol{u}_0\|_0 \\
&\lesssim h|\boldsymbol{u}-\boldsymbol{u}_0|_1+h^{s+1}|\boldsymbol{u}_0|_{s+1} \lesssim h\varepsilon^{1/2}\|\boldsymbol{f}\|_{0}+h^{s+1}|\boldsymbol{u}_0|_{s+1}.
\end{align*}
It follows from complex \eqref{eq:stokescdncfem}, the trace inequality, \eqref{errorestimateIhd1} and \eqref{regularity} that
\begin{align*}
\varepsilon^{2}\sum_{F\in\mathcal F_h^{\partial}}\frac{1}{h_F}\|\curl(\boldsymbol{u}-I_{h}^{gc}\boldsymbol{u})\|_F^2&=\varepsilon^{2}\sum_{F\in\mathcal F_h^{\partial}}\frac{1}{h_F}\|\curl\boldsymbol{u}-I_{h}^{d}(\curl\boldsymbol{u})\|_F^2 \\
&\lesssim \varepsilon^2 h^{1-r}|\curl\boldsymbol{u}|_1^r|\curl\boldsymbol{u}|_2^{1-r}\lesssim \varepsilon^{r-1/2} h^{1-r}\|\boldsymbol{f}\|_0.
\end{align*}
Thus we conclude \eqref{errorestimateIhd5} from \eqref{errorestimateIhd3}-\eqref{errorestimateIhd4} and the last two inequalities.
\end{proof}

\begin{theorem}
Let $\boldsymbol{u} \in \boldsymbol{H}_{0}(\grad \curl, \Omega)$ be the solution of the problem \eqref{quadcurl}, $\boldsymbol{u}_{0} \in \boldsymbol{H}_{0}(\curl, \Omega)$ be the solution of double curl equation \eqref{reduced}, and $\boldsymbol{u}_{h} \in \boldsymbol{W}_{h}$ be the solution of the mixed finite element methods \eqref{dismixfor1}-\eqref{dismixfor2}. Assume $\boldsymbol{u}_0\in H^{k+1}(\Omega;\mathbb R^3)$ and $\curl\boldsymbol{u}_0\in H^2(\Omega;\mathbb R^3)$.
Then we have the robust error estimates
\begin{equation}
\label{priorierrorrobustNitsche}
\interleave\boldsymbol{u}-\boldsymbol{u}_{h}\interleave_{\varepsilon,h} \lesssim \varepsilon^{r-1/2}h^{1-r}\|\boldsymbol{f}\|_0+ h^{k+1}|\boldsymbol{u}_0|_{k+1}+h^2|\curl\boldsymbol{u}_0|_2
\end{equation}
with $0\leq r\leq 1$, and
\begin{equation}
\label{priorierrorrobustNitscheu0}
\|\boldsymbol{u}_0-\boldsymbol{u}_{h}\|_{\varepsilon,h} \lesssim \varepsilon^{1/2}\|\boldsymbol{f}\|_0+ h^{k+1}|\boldsymbol{u}_0|_{k+1}+h^2|\curl\boldsymbol{u}_0|_2.
\end{equation}
\end{theorem}
\begin{proof}
Estimate \eqref{priorierrorrobustNitsche} holds from \eqref{priorierrorNitsche} and \eqref{consistencyerror3}-\eqref{errorestimateIhd5}. And \eqref{priorierrorrobustNitscheu0} follows from \eqref{priorierrorrobustNitsche} and regularities \eqref{eq:regularityu0}-\eqref{regularity}.
\end{proof}

\begin{remark}\rm
For a fixed $\varepsilon>0$, estimate \eqref{priorierrorrobustNitsche} tells us $\interleave\boldsymbol{u}-\boldsymbol{u}_{h}\interleave_{\varepsilon,h} \lesssim h$ by choosing $r=0$.
When $\varepsilon$ tends to zero, estimate \eqref{priorierrorrobustNitsche} implies $\interleave\boldsymbol{u}-\boldsymbol{u}_{h}\interleave_{\varepsilon,h} \lesssim h^2$ by choosing $r=1$, which is robust with respect to $\varepsilon$ and optimal in mesh size $h$.
\end{remark}

\section{Numerical results}\label{sec:numericalresults}
In this section, we perform numerical experiments to demonstrate the theoretical results of the mixed finite element methods \eqref{dismixfor1}-\eqref{dismixfor2} and \eqref{dismixforNitsche1}-\eqref{dismixforNitsche2}. Let $\Omega=(0,1)^3$. Choose the exact solution of double curl equation \eqref{reduced}
\[
\boldsymbol{u}_0 = \curl\begin{pmatrix}
0\\
0\\
\sin^2(\pi x)\sin^2(\pi y)\sin^2(\pi z)
\end{pmatrix},
\]
and let $\boldsymbol{f}=\curl^{2} \boldsymbol{u}_{0}$. The analytic expression of the exact solution $\boldsymbol{u}$ of the quad-curl singular perturbation problem \eqref{quadcurl} can not be explicitly given.
We take uniform triangulations on $\Omega$. Set $k=1$.

In the first experiment, we apply the lowest order Huang element in \cite{Huang2020} to solve the double curl problem \eqref{reduced}, i.e. mixed finite element method \eqref{dismixforNitsche1}-\eqref{dismixforNitsche2} for $\varepsilon=0$ with the finite element replaced by the Huang element.
Numerical errors $\|\boldsymbol{u}_0-\boldsymbol{u}_h\|_0$ and $\|\curl_h(\boldsymbol{u}_0-\boldsymbol{u}_h)\|_0$ with respect to $h$ are shown in Table~\ref{table:Huangerrorenery}.
We can see that both $\|\boldsymbol{u}_0-\boldsymbol{u}_h\|_0$ and $\|\curl_h(\boldsymbol{u}_0-\boldsymbol{u}_h)\|_0$ have no convergence rates, similarly as the Morley element method for Poisson equation \cite[Section 3]{NilssenTaiWinther2001}. Indeed, the convergence would deteriorate if using nonconforming finite elements to discretize the double curl problem, as illustrated in \cite[Section 7.9]{BoffiBrezziFortin2013}.
\begin{table}[htbp]
\caption{Errors $\|\boldsymbol{u}_0-\boldsymbol{u}_h\|_0$ and  $\|\curl_h(\boldsymbol{u}_0-\boldsymbol{u}_h)\|_0$ for double curl equation discretized by the lowest order Huang element.}\label{table:Huangerrorenery}
\centering
\begin{tabular}{lllll}
\hline\noalign{\smallskip}
$h$ & $\|\boldsymbol{u}_0-\boldsymbol{u}_h\|_0$ & order & $\|\curl_h(\boldsymbol{u}_0-\boldsymbol{u}_h)\|_0$ & order \\
\noalign{\smallskip}\hline\noalign{\smallskip}
$2^{-1}$ & 1.017E+00 & $-$     & 8.001E+00 & $-$ \\
$2^{-2}$ & 1.690E+00 & $-0.73$  & 9.575E+00 & $-0.26$ \\
$2^{-3}$ & 1.746E+00 & $-0.05$  & 1.068E+01 & $-0.16$ \\
$2^{-4}$ & 1.764E+00 & $-0.02$  & 1.101E+01 & $-0.04$ \\
$2^{-5}$ & 1.769E+00 & \;\;\,0.00  & 1.110E+01 & $-0.01$ \\
\noalign{\smallskip}\hline
\end{tabular}
\end{table}

Next we test the robustness of the mixed finite element method \eqref{dismixfor1}-\eqref{dismixfor2}.
Numerical errors $\|\boldsymbol{u}_0-\boldsymbol{u}_{h0}\|_0$,  $\|\curl(\boldsymbol{u}_0-\boldsymbol{u}_{h0})\|_0$ and $\|\boldsymbol{u}_0-\boldsymbol{u}_{h0}\|_{\varepsilon,h}$ for $\varepsilon=0$ and $\varepsilon=10^{-3}$ are shown in Table~\ref{table:mixfemeps0error} and Table~\ref{table:mixfemeps1E-3error} respectively.
By Table~\ref{table:mixfemeps0error} and Table~\ref{table:mixfemeps1E-3error}, numerically $\|\boldsymbol{u}_0-\boldsymbol{u}_{h0}\|_0\eqsim O(h)$,  $\|\curl(\boldsymbol{u}_0-\boldsymbol{u}_{h0})\|_0\eqsim O(h^{0.5})$ and $\|\boldsymbol{u}_0-\boldsymbol{u}_{h0}\|_{\varepsilon,h}\eqsim O(h^{0.5})$, which coincide with the sharp estimate \eqref{priorierrorrobustu0}. 
\begin{table}[htbp]
\caption{Errors $\|\boldsymbol{u}_0-\boldsymbol{u}_{h0}\|_0$,  $\|\curl(\boldsymbol{u}_0-\boldsymbol{u}_{h0})\|_0$ and $\|\boldsymbol{u}_0-\boldsymbol{u}_{h0}\|_{\varepsilon,h}$ for mixed finite element method \eqref{dismixfor1}-\eqref{dismixfor2} for $k=1$ and $\varepsilon=0$.}\label{table:mixfemeps0error}
\centering
\begin{tabular}{lllllll}
\hline\noalign{\smallskip}
$h$   & $\|\boldsymbol{u}_0-\boldsymbol{u}_{h0}\|_0$ & order & $\|\curl(\boldsymbol{u}_0-\boldsymbol{u}_{h0})\|_0$ & order & $\|\boldsymbol{u}_0-\boldsymbol{u}_{h0}\|_{\varepsilon,h}$ & order \\
\noalign{\smallskip}\hline\noalign{\smallskip}
$2^{-1}$ & 5.142E$-01$ & $-$   & 5.386E+00 & $-$   & 5.410E+00 & $-$ \\
$2^{-2}$ & 1.526E$-01$ & 1.75  & 2.573E+00 & 1.07  & 2.577E+00 & 1.07 \\
$2^{-3}$ & 5.390E$-02$ & 1.50  & 1.495E+00 & 0.78  & 1.496E+00 & 0.78 \\
$2^{-4}$ & 2.229E$-02$ & 1.27  & 9.933E$-01$ & 0.59  & 9.936E$-01$ & 0.59 \\
\noalign{\smallskip}\hline
\end{tabular}
\end{table}
\begin{table}[htbp]
\caption{Errors $\|\boldsymbol{u}_0-\boldsymbol{u}_{h0}\|_0$,  $\|\curl(\boldsymbol{u}_0-\boldsymbol{u}_{h0})\|_0$ and $\|\boldsymbol{u}_0-\boldsymbol{u}_{h0}\|_{\varepsilon,h}$ for mixed finite element method \eqref{dismixfor1}-\eqref{dismixfor2} for $k=1$ and $\varepsilon=10^{-3}$.}\label{table:mixfemeps1E-3error}
\centering
\begin{tabular}{lllllll}
\hline\noalign{\smallskip}
$h$   & $\|\boldsymbol{u}_0-\boldsymbol{u}_{h0}\|_0$ & order & $\|\curl(\boldsymbol{u}_0-\boldsymbol{u}_{h0})\|_0$ & order & $\|\boldsymbol{u}_0-\boldsymbol{u}_{h0}\|_{\varepsilon,h}$ & order \\
\noalign{\smallskip}\hline\noalign{\smallskip}
$2^{-1}$ & 5.143E-01 & $-$   & 5.386E+00 & $-$   & 5.412E+00 & $-$ \\
$2^{-2}$ & 1.529E-01 & 1.75  & 2.573E+00 & 1.07  & 2.583E+00 & 1.07  \\
$2^{-3}$ & 5.451E-02 & 1.49  & 1.495E+00 & 0.78  & 1.509E+00 & 0.78  \\
$2^{-4}$ & 2.358E-02 & 1.21  & 9.943E-01 & 0.59  & 1.027E+00 & 0.56  \\
\noalign{\smallskip}\hline
\end{tabular}
\end{table}

At last we numerically examine the modified mixed finite element method \eqref{dismixforNitsche1}-\eqref{dismixforNitsche2} by the Nitsche's technique.
Numerical errors $\|\boldsymbol{u}_0-\boldsymbol{u}_{h}\|_0$,  $\|\curl(\boldsymbol{u}_0-\boldsymbol{u}_{h})\|_0$ and $\|\boldsymbol{u}_0-\boldsymbol{u}_{h}\|_{\varepsilon,h}$ for $\varepsilon=0$ and $\varepsilon=10^{-3}$ are shown in Table~\ref{table:mixfemNitscheeps0error} and Table~\ref{table:mixfemNitscheeps1E-3error} respectively, from which we can observe that they all achieve the optimal convergence rate $O(h^2)$ numerically and agree with the theoretical estimate \eqref{priorierrorrobustNitscheu0}. 
\begin{table}[htbp]
\caption{Errors $\|\boldsymbol{u}_0-\boldsymbol{u}_{h}\|_0$,  $\|\curl(\boldsymbol{u}_0-\boldsymbol{u}_{h})\|_0$ and $\|\boldsymbol{u}_0-\boldsymbol{u}_{h}\|_{\varepsilon,h}$ for mixed finite element method \eqref{dismixforNitsche1}-\eqref{dismixforNitsche2} for $k=1$ and $\varepsilon=0$.}\label{table:mixfemNitscheeps0error}
\centering
\begin{tabular}{lllllll}
\hline\noalign{\smallskip}
$h$   & $\|\boldsymbol{u}_0-\boldsymbol{u}_{h}\|_0$ & order & $\|\curl(\boldsymbol{u}_0-\boldsymbol{u}_{h})\|_0$ & order & $\|\boldsymbol{u}_0-\boldsymbol{u}_{h}\|_{\varepsilon,h}$ & order \\
\noalign{\smallskip}\hline\noalign{\smallskip}
$2^{-1}$ & 4.660E$-01$ & $-$   & 4.553E+00 & $-$   & 4.576E+00 & $-$ \\
$2^{-2}$ & 1.193E$-01$ & 1.97  & 1.445E+00 & 1.66  & 1.450E+00 & 1.66  \\
$2^{-3}$ & 3.203E$-02$ & 1.90  & 4.189E$-01$ & 1.79  & 4.201E$-01$ & 1.79  \\
$2^{-4}$ & 8.277E$-03$ & 1.95  & 1.105E$-01$ & 1.92  & 1.108E$-01$ & 1.92  \\
\noalign{\smallskip}\hline
\end{tabular}
\end{table}
\begin{table}[htbp]
\caption{Errors $\|\boldsymbol{u}_0-\boldsymbol{u}_{h}\|_0$,  $\|\curl(\boldsymbol{u}_0-\boldsymbol{u}_{h})\|_0$ and $\|\boldsymbol{u}_0-\boldsymbol{u}_{h}\|_{\varepsilon,h}$ for mixed finite element method \eqref{dismixforNitsche1}-\eqref{dismixforNitsche2} for $k=1$ and $\varepsilon=10^{-3}$.}\label{table:mixfemNitscheeps1E-3error}
\centering
\begin{tabular}{lllllll}
\hline\noalign{\smallskip}
$h$   & $\|\boldsymbol{u}_0-\boldsymbol{u}_{h}\|_0$ & order & $\|\curl(\boldsymbol{u}_0-\boldsymbol{u}_{h})\|_0$ & order & $\|\boldsymbol{u}_0-\boldsymbol{u}_{h}\|_{\varepsilon,h}$ & order \\
\noalign{\smallskip}\hline\noalign{\smallskip}
$2^{-1}$ & 4.661E$-01$ & $-$   & 4.553E+00 & $-$   & 4.578E+00 & $-$ \\
$2^{-2}$ & 1.193E$-01$ & 1.97  & 1.445E+00 & 1.66  & 1.452E+00 & 1.66 \\
$2^{-3}$ & 3.204E$-02$ & 1.90  & 4.190E$-01$ & 1.79  & 4.224E$-01$ & 1.78 \\
$2^{-4}$ & 8.299E$-03$ & 1.95  & 1.130E$-01$ & 1.89  & 1.153E$-01$ & 1.87 \\
\noalign{\smallskip}\hline
\end{tabular}
\end{table}

\bibliographystyle{abbrv}
\bibliography{./refs}
\end{document}